\documentclass[a4paper,12pt]{amsart}
\usepackage[a4paper]{geometry}
\usepackage{hyperref}
\usepackage{amssymb}
\usepackage{ifthen}
\usepackage{graphicx}
\usepackage{pdfsync}
\usepackage{pifont}
\usepackage{bbm}
\newtheorem{theorem}[subsection]{Theorem}
\newtheorem{lemma}[subsection]{Lemma}
\newtheorem{proposition}[subsection]{Proposition}

\newtheorem{corollary}[subsection]{Corollary}

\DeclareMathOperator{\Tr}{Tr}

\DeclareMathOperator{\eig}{eig}

\renewcommand{\Re}{\Realpart}
\DeclareMathOperator{\Realpart}{Re}
\DeclareMathOperator{\Time}{Time}
\DeclareMathOperator{\Space}{Space}

\renewcommand{\emptyset}{\varnothing}

\newcommand{\X}{\mathcal{X}}
\newcommand{\M}{\mathcal{M}}

\newcommand{\F}{\mathcal{F}}

\newcommand{\finespace}{\mspace{1.0mu}}

\newcommand{\Prob}[1][]{\mathbb{P}\ifthenelse{\equal{#1}{}}{}{\finespace[\finespace#1\finespace]}}
\newcommand{\Probs}[1][]{\mathbb{P}^*\ifthenelse{\equal{#1}{}}{}{\mspace{-1mu}[\,#1\,]}}
\newcommand{\1}{\mathbf{1}}

\newcommand{\Real}{\mathbb{R}}
\newcommand{\Natural}{\mathbb{N}}
\newcommand{\N}{\mathcal{N}}

\newcommand{\OU}{\mathcal{DBM}}
\newcommand{\ADBM}{\mathcal{A-DBM}}
\newcommand{\bead}{\mathrm{Bead}}
\newcommand{\W}{\mathcal{W}}

\newcommand{\inv}{^{-1}}
\newboolean{dum}
\setboolean{dum}{false}
\newcommand{\comment}[1]{\ifthenelse{\boolean{dum}}{
{\par\noindent\Huge\ding{46}} \fbox{\parbox{10cm}{#1}}\par}{}}

\begin{document}
\comment{
$ $Id: article.tex,v 1.31 2011/05/16 13:29:49 enord Exp $ $
}
\author{Mark Adler}
\address{%
Department of Mathematics, Brandeis University,
Waltham, Mass 02454, USA.}
\email{adler@brandeis.edu}
\thanks{Adler: The support of a National Science Foundation grant 
\#~DMS-07-04271  is gratefully acknowledged.}

\author{Eric Nordenstam}
\address{%
Department of Mathematics, Universit\'e de Louvain, 1348
Louvain-la-Neuve,  Belgium}
\email{eric.nordenstam@uclouvain.be}
\thanks{%
Nordenstam: Supported by 
``Interuniversity Attraction Pole at UCL'' (Center of excellence): 
Nonlinear systems, stochastic processes and statistical mechanics (NOSY).}

\author{Pierre Van Moerbeke}
\address{%
Department of Mathematics, Universit\'e de Louvain, 1348
Louvain-la-Neuve, Belgium and Brandeis University,
Waltham, Mass 02454, USA.}
\email{pierre.vanmoerbeke@uclouvain.be \and
vanmoerbeke@brandeis.edu}
\thanks{%
van Moerbeke: 
The support of a National Science Foundation grant \#~DMS-07-04271 
is gratefully acknowledged. Also, a European Science Foundation 
grant (MISGAM), a Marie Curie Grant (ENIGMA), a FNRS grant and a 
``Interuniversity Attraction Pole'' grant are gratefully acknowledged.}

\title{The Dyson Brownian Minor Process}
\nocite{JoNo07,AdNoMo10}
\begin{abstract}
Consider an $n\times n$ Hermitean matrix valued stochastic process
$\{H_t\}_{t\geq 0}$ where the 
elements evolve according to Ornstein-Uhlenbeck 
processes. It is well known that
the eigenvalues perform a so called 
Dyson Brownian motion, that is they 
behave as Ornstein-Uhlenbeck processes conditioned
never to intersect.

In this paper we study not only the eigenvalues of the
full matrix, but also the eigenvalues of all the principal minors.
That is, the eigenvalues of the $k\times k$ in the upper left corner
of $H_t$. 
If you project this process to a space-like path 
it is a determinantal process and we compute the kernel. 
This kernel contains the well known GUE minor kernel, 
\cite{JoNo06,OkRe06}  and the 
Dyson Brownian motion kernel~\cite{FoNa98} as special cases.

In the bulk scaling limit of this kernel it is possible
to recover a time-dependent  generalisation of Boutillier's
bead kernel~\cite{Bo09}.

We also compute the kernel for a process of intertwined Brownian motions 
introduced by Warren in~\cite{Wa}. 
That too is a determinantal process along spacelike paths. 

\end{abstract} 
\maketitle
\section{Introduction}

In a classic paper Dyson~\cite{Dy62} introduced a dynamics
on random Hermitean matrices where each free matrix element 
evolves independently of all others. 
They each form an Ornstein-Uhlenbeck process, that 
is a Brownian motion with a drift toward zero. 
He successfully analysed the associated dynamics of the eigenvalues.
Even the fact that the eigenvalues form a Markov process is 
highly non-trivial. 

This has become one of the most well studied models of random
matrix theory. 
It is beyond the scope of this paper to completely survey the 
literature but some notable results are these.
It has been found to be a limit of discrete random walks
conditioned never to intersect, for which correlation kernels have been found; see~\cite{FoNa98,Jo02,Jo03,Jo05a,Jo05b,EiKo08,KaTa02,KaTa03}. 
Other work includes~\cite{Be08,Sp86,TrWi04, AvM05}. In particular, in ~\cite{TrWi04, AvM05}, partial differential equations were derived for the Dyson process and related processes.

Actually there are two processes that have been called
Dyson Brownian motion. 
Let a Hermitean  matrix $B_t$ evolve according to
the transition density, for $s<t$, 
\begin{equation}
\label{eqn:dyson-bm1}
\Pr ( B_t \in dB | B_s = \bar B) 
= C e^{-Tr (B-q_{t-s} \bar B)^2 / (1-q_{t-s}^2)} 
\end{equation}
where $q_{t-s} = e^{-(t-s)} $ and $C$ is the normalisation constant
that makes this a probability density. 
Then this is a stationary process and its
stationary measure is called the \emph{Gaussian Unitary Ensemble} (GUE), 
see~\cite[chapter 9]{Me06}. 
For an~$n\times n$ GUE matrix~$B$, 
\begin{equation}
\Pr ( B \in dB) =
 C e^{ -\Tr B^2} \, dB.
\end{equation}
What this boils down to is that 
the elements on the diagonal are independent Gaussians with mean 0
and variance $\frac12$. 
For the off-diagonal elements the real and imaginary parts
are independent  Gaussians with mean 0 and variance $\frac14$.
For an~$n\times n$ Hermitean matrix $B$ let 
$\eig B = (\lambda_1 <  \cdots <  \lambda_n)$
denote the vector of eigenvalues of $B$. 
The distribution of  the eigenvalues is exactly
\begin{equation}
\Pr ( \eig B \in d\lambda) =
 C \prod_{1\leq i<j\leq n} (\lambda_i-\lambda_j) ^2 
 \prod_{i=1}^n e^{-\lambda_i ^2}\,d\lambda_i
\end{equation}
where $C$ is a normalisation constant, see~\cite[chapter 3]{Me06}. 
The transition density for the eigenvalues is 
then, by the Harish-Chandra formula,
\begin{equation}
\Pr ( \eig B_t \in d\lambda \, |\, \eig B_s = \bar \lambda) =
 C\frac{\Delta (\lambda) }{\Delta(\bar\lambda)}
 \det [ e^{-(\lambda_i - e^{-(t-s)}  \bar\lambda_j)^2 /(1-e^{-2(t-s)}) } ] _{i,j=1}^n\prod_{i=1}^n dx_i .
\end{equation}
Here, $\Delta$ denotes the Vandermonde determinant. 
See~\cite{Jo05a} for a readable overview. 
This expression is a  Doob $h$-transform 
of a Karlin-McGregor determinant, see~\cite{Do84,KaMcG59}.
That means that this can be interpreted probabilistically
as $n$ Ornstein-Uhlenbeck processes evolving in time 
conditioned never to intersect. 

Of course it is very natural to consider $n$ 
pure Brownian motions $(x_1(t)$, \dots, $x_n(t))_{t\in\Real} $  
conditioned never to intersect. 
The transition density for that process would be
\begin{equation}
\label{eqn:dyson-bm3}
\Pr ( x(t)  \in dx \, |\, x(s) = \bar x) =
 C\frac{\Delta (x) }{\Delta(\bar x )}
 \det [ e^{-(x_i - \bar x_j)^2 /(t-s) } ] _{i,j=1}^n\prod_{i=1}^n dx_i.
\end{equation}
This process has also been called Dyson's Brownian motion,
but is not stationary. 
It can be realised by a Hermitean matrix where
the elements evolve as Brownian motions. 

For the purpose of this article consider the following
model which we shall call the \emph{Dyson Brownian minor process} 
or (DBM process). 
Let  $(B_t)_{t\in\Real^+}$ be an $N\times N$  Hermitean
 matrix-valued stochastic  process started at 
$t=0$ with $B_0$ given by the GUE distribution.
Let the process evolve with transition density given by~\eqref{eqn:dyson-bm1}.
For $n=1$, \dots, $N$ let $B^{(n)}_t $ be the $n\times n$ 
submatrix in the upper left corner (principal minor) of $B_t$. 
We are interested in all the 
$\binom{N+1}{2}$ eigenvalues of $B^{(n)}_t$ for $n=1$, \dots, $N$. 

If $\lambda$ is an eigenvalue of $B_t^{(n)}$ then we shall 
say that there is a particle at $(n, t, \lambda)$. 
In this way we can think of the DBM process as a point process,
that is a measure on configurations of points or particles 
on the space $\{1,\dots, N\} \times \Real \times \Real$.
It turns out that this is, in the terminology of~\cite{BoFe08},
 a determinantal point process along space-like paths. 
More precisely, that means the following.

For notation we shall write that 
$$
(n,t) < (n',t') := \begin{cases}
\mathtt{true} & \text{if $n>n'$}\\
\mathtt{true} & \text{if $n=n'$ and $t < t'$ }\\
\mathtt{false} & \text{otherwise.}
\end{cases}
$$
and 
$$ (n,t) \geq (n',t')  := \neg ( (n,t) < (n',t') ).
$$ 
\begin{theorem}
\label{thm:main-theorem-1}
Take a sequence $\{(n_i, x_i, t_i)\}_{i=1}^k$ of levels, positions
and times. 
Let them follow a space like path, which means that 
\begin{equation}
\label{eqn:spacelike1}
0\leq t_1\leq t_2 \leq \cdots \leq t_k, 
\end{equation}
\begin{equation}
\label{eqn:spacelike2}
n_1\geq n_2 \geq \cdots \geq n_k.
\end{equation}
Then the density of the event
that there is a particle  at time $t_i$ on level 
$n_i$ at position $x_i$ in the Dyson Brownian minor process is
\begin{equation}
\rho((n_1, x_1, t_1), \dots, (n_k, x_k, t_k))
=
\det [ K^\OU( (n_i, x_i, t_i), (n_j, x_j, t_j))]_{i,j=1}^k
\end{equation}
where
\begin{multline}
K^\OU ((n, x, t), (n', x', t')) = \\
\begin{cases}
\displaystyle
\sum_{l=-\infty}^{-1}
\sqrt{
 \frac{(n' +l )!}{(n + l)!}
 }
e^{-l(t'-t) }
h_{n+l}^{*}(x)
h_{n'+l}^{*}(x') 
e^{-(x')^2}, & \text{for $(n,t) \geq (n',t') $,}\\
\displaystyle
-\sum_{l=0}^{\infty}
\sqrt{
 \frac{(n' + l )!}{(n + l)!}
 }
e^{-l(t'-t) }
h_{n+l}^{*}(x)
h_{n'+l}^{*}(x') 
e^{-(x')^2} , & \text{for $(n,t) <(n',t')$.}
\end{cases}
\end{multline}
Furthermore 
\begin{multline}
\label{eqn:DB-kernel}
K^\OU ((n, x, t), (n', x', t')) = 
- \phi^\OU ((n, x, t), (n', x', t'))
\\
+2^{(n'-n)/2}  \frac{2}{(2\pi i)^2} 
\int_\gamma du \int_\Gamma dv  
\frac{v^{n'}}{u^n}
\frac{ e^{-u^2 + 2ux + v^2 - 2vx'} }{e^{-(t'-t)} v - u} 
\end{multline}
where
\begin{multline}
\label{eqn:DB-kernel2}
\phi^\OU ((n, x, t), (n', x', t')) = \\ \begin{cases}
\displaystyle
2^{(n-n')/2}  e^{n' (t'-t) } \int_\Real H^{n-n'} (x-y) p^{*} _{t'-t}( y,x') \,dy, & 
\text{if $(n,t) < (n',t')$,}\\
0, & \text{otherwise.}
\end{cases}
\end{multline}
The contours of integration are such that $\gamma $
encloses the pole at the origin and $\Gamma $ goes from $-i\infty$
to $i\infty $ in such a way that $|u|<|v|$ always,
 see Figure~\ref{fig:contours}.
\end{theorem}

Here $h^*_n$, for $n=0$, 1, \dots, is the normalised Hermite polynomial of order $n$,
see Section~\ref{sec:dyson}. 
$h^*_n$ for $n=-1$, $-2$, \dots, are defined to be  zero. 
$p^*$ is the transition density of  an Ornstein-Uhlenbeck 
process, see~\eqref{eqn:pstar-def}. $H^n$ is the $n$th anti-derivative
of the Dirac delta function, see~\eqref{eqn:heaviside}. 
This Theorem will be proved in Section~\ref{sec:finale}.

The term space-like path for a path in space-time 
satisfying~\eqref{eqn:spacelike1} and~\eqref{eqn:spacelike2}
was coined in~\cite{BoFe08} and, while the 
reason for using that name is not made clear, the terminology has
become standard. 

It is quite clear that this process, along a space-like path, is 
a Markov process. It is known from Dyson~\cite{Dy62} that
this process is Markovian on a fix level, that is for constant~$n$. 
Baryshnikov~\cite{Ba} observed that it is a Markov process
for fixed time going down one level (say from $n$ to $n-1$). 
A space-like path is a combination of steps in  time 
and steps going down one level, i.e. it is a combination of consecutive
 Markov steps.

One beautiful construction which is due to Warren~\cite{Wa}
is the following. 
Start a 1-dimensional Brownian motion, say $(x(t))_{t\geq 0} $ at the origin  
at time $t=0$. 
Then start two new processes, say $(y_1 (t))_{t\geq 0}$ and 
$ (y_2(t))_{t\geq 0}$ respectively,
one above and one below, respectively,  the old one. 
They evolve as Brownian motions except that they are pushed
up and down, respectively, by $x$. 
For details see~\cite{Wa}. 
It is then a Theorem that $y_1$ and $y_2$ together form a 
Dyson Brownian motion in the sense that their 
transition density is of the same form as~\eqref{eqn:dyson-bm3}.
The process can be continued: one can start three processes
above, between and below $y_1$ and $y_2$. These three
will then be a Dyson Brownian motion of three particles and so on. 
This process occurs as a scaling limit in the study of
a certain random tiling model~\cite{No10}. 
To reduce the amount of numerical factors floating around we 
shall in this paper consider the Warren process to be driven by 
Brownian motions with variance $t/2$ instead of standard Brownian motion. 
This is just a rescaling that is not important. 

For this model we can show a result analogous to that of 
Theorem~\ref{thm:main-theorem-1}.

\begin{theorem}
\label{thm:main-theorem-2}
Take a sequence $\{(n_i, t_i, x_i)\}_{i=1}^k$ of times 
as in the previous Proposition. 
Then the density of the event
that there is a particle  at time $t_i$ on level 
$n_i$ at position $x_i$ in Warren's process is
\begin{equation}
\rho((n_1, t_1, x_1), \dots, (n_k, t_k, x_k))
=
\det [ K^\W( (n_i, t_i, x_i), (n_j, t_j, x_j))]_{i,j=1}^k
\end{equation}
where
\begin{multline}
K^\W ((n, x, t), (n', x', t')) = \\
\begin{cases}
\displaystyle
\frac{1}{\sqrt{t}}
\sum_{l=-\infty}^{-1}
\sqrt{
 \frac{(n' +l )!}{(n + l)!}
    \left(
      \frac{t}{t'}
    \right)^l
}
h_{n+l}^{*}(x/\sqrt{t})
h_{n'+l}^{*}(x'/\sqrt{t'}) 
e^{-(x')^2/t'}, & \text{for $(n,t) \geq (n',t') $,}\\
\displaystyle
-\frac{1}{\sqrt{t}}
\sum_{l=0}^{\infty}
\sqrt{
 \frac{(n' +l )!}{(n + l)!}
    \left(
      \frac{t}{t'}
    \right)^l
}
h_{n+l}^{*}(x/\sqrt{t})
h_{n'+l}^{*}(x'/\sqrt{t'} ) 
e^{-(x')^2/t'}, & \text{for $(n,t) <(n',t')$.}
\end{cases}
\end{multline}

Furthermore 
\begin{multline}
2^{\frac12 (n-n')}
K^\W ((n, x, t), (n', x', t')) = 
- \phi^\W ((n, x, t), (n', x', t'))
\\
+  \frac{2}{(2\pi i)^2} 
\frac{t^{n/2} (t')^{n'/2} }{\sqrt{t}} 
\int_\gamma du \int_\Gamma dv  
\frac{v^{n'}}{u^n}
\frac{ e^{-u^2 + 2ux/\sqrt{t} + v^2 - 2vx'/\sqrt{t'} } }{t v - u} 
\end{multline}
where
\begin{equation}
\phi^\W ((n, x, t), (n', x', t')) = \begin{cases}
\displaystyle
\sqrt{\frac{(t')^{n'}}{t^n}} \int_\Real H^{n-n'} (x-y) p_{t'-t}( y,x') \,dy, & 
\text{if $(n,t) < (n',t')$,}\\
0, & \text{otherwise.}
\end{cases}
\end{equation}
The contours of integration are such that $\gamma $
encloses the pole at the origin and $\Gamma $ goes from $-i\infty$
to $i\infty $ in such a way that $|u|<|v|$ always,
 see Figure~\ref{fig:contours}.
\end{theorem}

Again $h^*_n$ is the normalised Hermite polynomial of order $n$,
see Section~\ref{sec:warren}
and $H^n$ is the $n$th anti-derivative
of the Dirac delta function, see~\eqref{eqn:heaviside}. 
$p$ is the transition density of  a Brownian motion, 
see~\eqref{eqn:p-def}. 
This Theorem will be proved in Section~\ref{sec:finale}.

Note that although the kernels $K^{\W}$ and
$K^{\OU}$ are just a change of variables from each other,
the underlying processes are  different in an essential way. 
As noted in~\cite{AdNoMo10}, the difference of the eigenvalues of successive
levels are pushed apart by a constant drift when they are close. 
In the construction due to Warren in~\cite{Wa} the difference of the particles
on successive levels behave like the absolute value of a Brownian
motion when they are close. 
It is to be remarked that if the minor process was constructed with 
Brownian motions replacing the Ornstein-Uhlenbeck prcesses,
the marginals along space-like paths would agree with the Warren process.

No article on random matrices is complete without a  scaling limit,
so let us do one of those. 
In~\cite{Bo09}, Boutillier introduced a one-parameter family of
 models which are point processes on $\Natural\times \Real$. 
On each individual copy of $\Real$, it specialises
to a determinantal process with the sine kernel which 
is so prevalent in all branches of 
 random matrix theory, see~\cite{Me06}. 
Furthermore, on successive lines the particles interlace. 
By that I mean that if there are particles at $(n, x_1)$
and  $(n, x_2)$, then there exists a particle $(n+1, y)$ 
such that $x_1<y<x_2$ almost surely. 
As a scaling limit of $ K^\OU$ above we recover a kernel
which specialises to the Boutillier Bead kernel at a fixed time. 
One way to interpret this is to imagine all the particles
in Boutillier's model moving in time in such a way that 
at each fixed time the picture looks like the original Bead kernel model. 

\begin{theorem}
\label{thm:main-theorem3}
Let $a$ be a real number on the interval $(-1,1)$.
In the bulk scaling limit around $a\sqrt{2N}$
the Dyson Brownian minor kernel
converges  to a time dependent Bead kernel with parameter $a$.
More precisely,
\begin{multline}
K^\bead_a ( (n, x, t), (n', x', t')) = 
\lim_{N\rightarrow\infty}
e^{-N(t'-t)} (4N)^{\frac12 (n-n')}
(2N)^{-\frac12} \times \\ 
\times K^\OU ((N+n,  \sqrt{2N}a + \frac{x}{\sqrt{2N}}, \frac{t}{2N}), 
(N+n',  \sqrt{2N}a + \frac{x'}{\sqrt{2N}}, \frac{t'}{2N}))
\end{multline}
for 
\begin{multline}
\label{eqn:bead}
K^\bead _a ( (n, x, t), (n', x', t')): =
\\
-\phi^\bead _a ( (n, x, t), (n', x', t'))
+ \frac{1 }{2\pi i} \int_{u_-}^{u_+} 
u^{n'-n} e^{\frac12 (t'-t)(u^2 -2a u )  + u(x-y) } \, du
\end{multline}
where
\begin{multline}
\label{eqn:bead2}
\phi^\bead _a ( (n, x, t), (n', x', t')) = \\ \begin{cases} 
2^{\frac 12(n-n')}  \int_\Real H ^{n-n'} (x-y) p_{\frac12(t'-t)} (y,x'-a (t'-t) ) \, dy &
\text{if $(n,t) < (n',t')$}\\
0 & \text{otherwise.}
\end{cases}
\end{multline}
The limit holds uniformly on compact sets
and the contour of integration in~\eqref{eqn:bead}
is the straight line between $u_-$ and $u_+$
where
\begin{equation}
\label{eqn:saddle-points}
u_\pm = a \pm i\sqrt{1-a^2}
\end{equation}
are two points on the unit circle. 
\end{theorem}

(The topology used to define compact sets on $(\mathbb{N}\times\mathbb{R}^2)^2$ 
is the product topology of discrete topology on $\mathbb{N}$ and 
Euclidean topology on $\mathbb{R}$.)

 Remark  that~\eqref{eqn:bead2} contains the 
transition density  $p$ of Brownian motion, defined in~\eqref{eqn:p-def},
rather than $p^*$, defined in~\eqref{eqn:pstar-def}. 
This Theorem is proved in Section~\ref{sec:bead}.
Remark too that, since the $K^\OU$ kernel could  only be used
along spacelike paths, the same is true for our time dependent
Bead kernel.
Finally notice that the kernel $K^\bead _a $ coincides
with Boutillier's kernel in~\cite{Bo09} with parameter $a$ in the 
special case $t=t'$.

Of course convergence of the kernel in this case  does not imply 
convergence of the processes, since the kernel only says something 
about the behaviour on space-like paths. 
Nor does this Theorem give any hint as to how one might 
construct such a dynamical version of a Bead process.
However, specialising the Theorem above to $t=t'$ leads to 
the following Corollary.
Though by no means unexpected, this result has to our knowledge 
not previously appeared in the litterature. 
\begin{corollary}
The GUE Minor process, defined in~\cite{JoNo06}, converges
in the same bulk scaling limit to Boutillier's bead process, 
defined in~\cite{Bo09}.
\end{corollary}
\begin{proof} 
The kernel $K^\OU$ specialised to $t=t'$ is exactly the 
GUE Minor kernel and $K^\bead _a $ with $t=t'$ is Boutillier's kernel.
Uniform convergence on compacts for the kernels is necessary 
for process convergence. 
\end{proof}

The plan of the paper is to first, in Sections~\ref{sec:point-processes},
\ref{sec:l-ensemble} and~\ref{sec:tips},
outline the necessary  basic theory about point processes. 
This is mostly a verbose summary of~\cite{BoRa,BoFePrSa}. 
Section~\ref{sec:definitions} is devoted to 
computing certain convolution equalities and setting up 
clever notation so that Theorems~\ref{thm:main-theorem-1} 
and~\ref{thm:main-theorem-2} can be proved at the same time, 
performing the computation only once. 
In Section~\ref{sec:kernel} the actual computation is performed
and the article is rounded of by the asymptotic analysis in 
Section~\ref{sec:bead}.

The fact that the GUE minor kernel kan be extended to a dynamic
version in this way begs the question whether something similar
can be done with the Anti-symmetric GUE minor kernel 
from~\cite{De08,Def08,FoNo09}.
In recent works, see~\cite{BoFePrSaWa09,KoSch09},  random walks
conditioned  to stay in Weil chambers of the form
\begin{align}
0<x_1<x_2< \cdots < x_n
\intertext{and}
|x_1|<x_2<\cdots < x_n
\end{align}
and their diffusion limits have been analysed. 
It is reasonable to believe that such processes could be realised
by Ornstein-Uhlenbeck dynamics on Anti-symmetric
purely imaginary matrices of odd respectively even size.
If so then it appears this model could be analysed with the same 
tools used in this paper and would lead to a Theorem similar 
to~\ref{thm:main-theorem-1} but with and Anti-symmetric
Dyson Brownian minor kernel
\begin{multline}
K^\ADBM ((n, x, t), (n', x', t')) = \\
\begin{cases}
\displaystyle
\sum_{l=-\infty}^{-1}
\sqrt{
 \frac{(n' + 2l )!}{(n + 2l)!}
 }
e^{-2l(t'-t) }
h_{n+2l}^{*}(x)
h_{n'+2l}^{*}(x') 
e^{-(x')^2}, & \text{for $(n,t) \geq (n',t') $,}\\
\displaystyle
-\sum_{l=0}^{\infty}
\sqrt{
 \frac{(n' + 2l )!}{(n + 2l)!}
 }
e^{-2l(t'-t) }
h_{n+2l}^{*}(x)
h_{n'+2l}^{*}(x') 
e^{-(x')^2} , & \text{for $(n,t) <(n',t')$.}
\end{cases}
\end{multline}
That is beyond the scope of this paper, but this kernel
is in~\cite{BoFeSa09} recovered as a scaling limit in a certain model
related to the totally asymmetric simple exclusion process (TASEP).

While this article was being prepared it came to the attention 
of the authors that Patrik Ferrari and Ren\'e Frings~\cite{FeFr10}
were working on related problems. They prove an analog 
of~\ref{thm:main-theorem-1} for matrices whose elements evolve
as Brownian motions and also for the Laguerre ensemble.
\comment{
$ $Id: l-ensemble.tex,v 1.47 2011/10/13 14:21:57 enord Exp $ $
}
\comment{We begin with a not very brief review of 
the theory of determinantal processes. This may or may 
not make it to the final article.}

\section{Point processes}
\label{sec:point-processes}
Let $\Lambda$ be a complete, separable metric space with some reference measure
$\lambda$. Say $\mathbb{R}$ with the Lebesgue measure or $\mathbb{Z}$ with
counting measure. Let $M(\Lambda)$ be the set of integer valued and
locally finite measures on $\Lambda$. A \emph{point process} $X$ on $\Lambda$ 
is a measure on $M(\Lambda)$. It is beyond the scope of this article 
to give a complete overview of the theory of point processes, but some 
results which we use are detailed here. 

A point process can be represented as 
\begin{equation}
X = \sum_{i\in I} \delta_{x_i}
\end{equation}
where $(x_i)_{i\in I} $ are random variables which we shall refer to as the 
points or the particles of $X$. 
Think of this as a random configuration of points or particles on the 
space $\Lambda$. In this paper we shall only consider point processes 
which are simple, i.e. all $x_i$ are distinct.

To work with point processes it is convenient to define the so called 
\emph{correlation functions}. 
For $n=1, 2, \dots$, these are functions 
$\rho_n : \Lambda^n\rightarrow \mathbb R$.
When $\Lambda = \mathbb Z$  and $\lambda$ is counting measure
then 
\begin{equation*}
\rho_n(x_1, \dots, x_n) = \mathbb P[\text{There is a particle at each position $x_i$, for $i=1$, \dots, $n$}].
\end{equation*}
When $\Lambda = \mathbb R$  and $\lambda$ is Lebesgue measure
then 
\begin{equation*}
\rho_n(x_1, \dots, x_n) = \lim_{\epsilon\rightarrow 0}
\frac{
\mathbb P[\text{There is a particle in each of $[x_i, x_i+\epsilon)$, for $i=1$, \dots, $n$}]}{\epsilon^n}.
\end{equation*}
More generally, see~\cite{Jo06}, one can define correlation functions
by saying that for simple, measurable functions $\phi$ of bounded support, 
the point process satisfies
\begin{equation}
\label{eqn:correlationfunctions}
\mathbb E [\prod_{i\in I} (1+\phi(x_i) ) ] =
1+\sum_{n=1}^\infty \frac{1}{n!} \int_{\Lambda^n} 
\prod_{j=1}^n \phi(y_j) \rho_n(y_1,\dots,y_n)\, d\lambda^n (y).
\end{equation}

A \emph{determinantal point process} is a point process whose correlation 
functions have the special form 
\begin{equation}
\rho_n(y_1, \dots, y_n ) = \det [ K(y_i, y_j) ]_{i,j=1}^n,
\end{equation}
for some function $K:\Lambda^2 \rightarrow \mathbb C$. This is a very special 
and simple situation since all information about the point process is encoded
in this function $K$ of two variables which is called the correlation kernel.
Nonetheless processes of this kind are commonplace in mathematics today
arising from such diverse sources as tilings with rhombuses or dominoes 
of regions in the plane, random walks, eigenvalues of unitarily invariant 
random matrices and, as shown in the next section, so called L-ensembles.
Indeed, the main theorems of this paper, Theorems~\ref{thm:main-theorem-1} 
and~\ref{thm:main-theorem-2}, state that certain processes are indeed 
determinantal point processes. 
Also note that with correlations functions of this form,
the right hand side of of~\eqref{eqn:correlationfunctions} turns out 
to be the definition of the Fredholm determinant of the integral operator
with kernel $K(x,y)\phi(y)$, again see~\cite{Jo06}. 

Asymptotic analysis of such point processes can be performed by working with
the kernels only. 
\begin{proposition}[Proposition 2.1 in~\cite{JoNo06}]
\label{thm:determ-point-proc}
Let  $X^1$, $X^2$, \dots, $X^N$, \dots
be a sequence of determinantal point processes, 
and let $X^N$ have correlation kernel $K^N$ satisfying 
\begin{enumerate}
\item
$K^N\rightarrow K$, $N\rightarrow\infty$ pointwise, for some function $K$,
\item the $K^N$ are 
 uniformly bounded on compact sets in $\Lambda^2$ and
\item 
For $C$ compact, there exists some number $n=n(C)$ such that
\begin{equation*} 
\det[K^N(x_i,x_j)]_{1\leq i,j\leq m} = 0
\end{equation*} if $m\geq n$.
\end{enumerate}
Then there exists some determinantal point process $X$ with correlation
kernel $K$ such that  $X^N\rightarrow X$ weakly, $N\rightarrow \infty$.
\end{proposition}

\section{Introduction to L-ensembles}
\label{sec:l-ensemble}
This section summarises the exposition in~\cite{BoRa}.
The reader who wishes to pursue the subject of determinantal point
processes will find~\cite{Bor09,Ma} illuminating.
\subsection{Measure theory}
Let~$(\Omega$, $\Prob$, $\F)$ be a 
\emph{discrete probability space}.
That is, $\Omega$ is a finite set, $\mathcal{F}=2^\Omega$ is the
$\sigma$-algebra of subsets of~$\Omega$. 
Furthermore, $\mathbb{P}:\mathcal{F}\rightarrow \mathbb{R}$ is a measure
satisfying 
\begin{enumerate}
\item $ \Prob[\Omega]= 1$,
\item $ \Prob[\emptyset]= 0$,
\item $ E_1$, $E_2 \in \mathcal{F}$ and $ E_1\cap E_2 = \emptyset$
implies 
$\Prob[E_1 \cup E_2 ] = \Prob[E_1]+\Prob[E_2]$. 
\end{enumerate}
The last property is called additivity.

An element~$E\in \mathcal{F}$ is called an \emph{event} 
and~$\Prob[E]$ is called the probability of~$E$.
To sample the distribution~$\mathbb{P}$ means picking an
element~$\omega \in \Omega$.  
Since we are working with a finite space,
$\Prob[E]$ can be decomposed as a sum over
singleton sets, 
\begin{equation*}
\Prob[E\finespace]= \sum_{\omega \in E} \Prob[\{\omega\}].
\end{equation*}
\subsection{Point processes.}
We now specialise and consider probability spaces of the following
form. Take a finite set $\X$. 
A \emph{point process\/} on $\X$ is a probability 
space~$(\Omega=2^\X, \mathbb{P}, \mathcal{F}=2^\Omega=2^{2^\X})$.
To sample the point process means to pick an element of~$2^\X$, 
i.e., a subset of~$\X$. 
For compatibility with~\cite{BoRa}, 
we will use uppercase letters at the end of the alphabet,
$X$,~$Y$,~\dots, when speaking about elements of~$\Omega$. 

Given an~$X\in \Omega$, let~$E_\X(X)$
be the event  
$$
\mathcal{F} \ni E_\X(X)= 
\bigcup_{X \subseteq Y\subseteq \X}
\{ Y \}.
$$

The probabilistic interpretation is that~$E_\X(X)$
is the event that all~$x\in X$ are in the chosen set.
Note that~$E_\X : \Omega \rightarrow \mathcal{F}$.

\subsection{L-ensembles.}
Let us specialise even more. 
Take a matrix~$L$ of size~$|\X|\times|\X|$. 
We shall index the rows and columns of this matrix by $\X$.
The \emph{L-ensemble} on $\X$ is  a point process
$(\Omega=2^\X, \Prob, \F=2^{2^\X})$ 
such that, for $X\subseteq \X$, 
\begin{equation}
\label{eqn:L-measure}
\Prob[\{X\}] = \frac{\det L_X}{\det(\1 + L)}.
\end{equation}
Here $L_X$ means pick out those rows and columns that 
correspond to $X$, this giving a $|X|\times|X|$ matrix.
Also, $\1$ is the identity matrix of appropriate size.
For sets in $\F$ which are not singletons, the measure 
$\Prob$ is defined by the additivity property.

This is only a  well defined probability if the 
expression (\ref{eqn:L-measure}) is always positive, 
for example if $L$ is positive definite. The fact that 
the probabilities sum up to one is guaranteed by the following 
well known formula, the Fredholm expansion of a determinant.
\begin{lemma}
\label{thm:determinant-expansion}
Let  $M$ be a matrix whose rows and columns are indexed by 
the finite set $\X$. Then
\begin{equation*}
\det (\1 + M) = \sum_{X\subseteq \X} \det M_X,
\end{equation*}
with the understanding that the determinant of the empty matrix is $1$.
\end{lemma}
We do not show this, it can for example be done by induction 
over the size of the matrix. To convince the reader of the validity 
and triviality of the above lemma, let us see what happens with 
a $2\times2$ matrix.
\begin{multline*}
\det \begin{bmatrix} 1+a & b \\ c & 1+d \end{bmatrix} =
\det \begin{bmatrix} 1 & b \\0 & 1+d \end{bmatrix} +
\det \begin{bmatrix} a & b \\ c & 1+d \end{bmatrix} \\
=\det \begin{bmatrix} 1 & 0 \\0 & 1 \end{bmatrix} +
\det \begin{bmatrix} 1 & b \\0 & d \end{bmatrix} +
\det \begin{bmatrix} a & 0 \\ c & 1 \end{bmatrix} +
\det \begin{bmatrix} a & b \\ c & d \end{bmatrix} 
\ifthenelse{\lengthtest{\textwidth<15cm}}{\\}{}
=1 +
\det \begin{bmatrix}  d \end{bmatrix} +
\det \begin{bmatrix} a \end{bmatrix} +
\det \begin{bmatrix} a & b \\ c & d \end{bmatrix} 
\end{multline*}
The first equality comes from the fact that the determinant
is linear in the first and second column.
The second equality comes from expanding along rows. 

\begin{theorem}
Let $K=L(\1+L)^{-1}$. Then for all $X\in 2^\X$, 
\begin{equation*}
\Prob[E_\X(X)] = \det K_X.
\end{equation*}
\end{theorem}
The matrix $K$ is frequently called the correlation kernel
of $\Prob$.

\begin{proof}
Let $\1_{(X)} $ be the identity matrix with 
the ones corresponding to elements in $\bar X:=\X\setminus X$
set to zero.
\begin{align*}
\Prob[E_\X(X)]&=
\sum_{X \subseteq Y \subseteq \X}
\Prob[\{Y\}] \\
&=
\det (\1+L)^{-1}  
\sum_{X \subseteq Y \subseteq \X}
\det L_Y \\
&=\det (\1_{(\bar X)} + L)(\1+L)^{-1}\\
&=\det (\1+L - \1_{(X)}) (\1+L)^{-1}\\
&=\det (\1 - \1_{(X)} (\1+L)^{-1})\\
&=\det [\1 -  (\1+L)^{-1}]_X\\
&=\det [L(\1+L)^{-1}]_X
\end{align*}
The third equality is due to a variant of 
Lemma \ref{thm:determinant-expansion}.
\end{proof}

\subsection{Projection on a subspace.}
Given an L-ensemble on $\X$, let us take an arbitrary fixed 
subset $\N\subset \X$. Define it's conjugate $\bar \N:=\X \setminus \N$.
We want to study a 
certain projection $\Probs$ of $\Prob$ to $2^{2^\N}$.
This shall give us a smaller point process 
$(\Omega^*=2^\N, \Probs, \F^*=2^{2^\N})$, specified by 
\begin{equation*}
\Probs[\{D\}] = 
\frac{\Prob [\{D \cup \bar \N\,\}]}{\Prob[E_\X(\bar \N\,)]}.
\end{equation*}
Again $\Probs$ is defined  by additivity for 
events that aren't singletons. 
This should be thought of as a conditional probability. 
Compute for example 
\begin{align*}
\Probs[E_\N(D)] &= \Probs[\cup_{D\subseteq F \subseteq \N} \{F\,\}]\\
&=\sum_{D\subseteq F \subseteq \N} \Probs[ \{F\}] \displaybreak[0]\\
&=\sum_{D\subseteq F \subseteq \N} \frac{\Prob[ \{F\cup \bar \N\}]}{\Prob[E_\X(\bar \N)]}\\
&=\frac{\Prob[E_\X(D \cup \bar \N)]}{\Prob[E_\X(\bar \N)]}\\
&=\frac{\Prob[E_\X(D)  \cap E_\X(\bar \N)]}{\Prob[E_\X(\bar \N)]}.
\end{align*}
\begin{theorem}
\label{thm:projected-kernel}
Let $K^*=\1_{(\N)} - (\1_{(\N)}+L)^{-1} |_\N$. Then for all $X\in \F^*$, 
\begin{equation*}
\Probs[E_\N(X)] = \det K^*_X.
\end{equation*}
\end{theorem}
Here, $|_\N$ means pick out those rows and colums that correspond to $\N$.
For proof see~\cite{BoRa}.

\subsection{Eynard-Mehta Theorem}
\label{sec:Eynard-Mehta}
Again let us specialise. We are interested in studying a 
point process on the space 
$\N=\X^{(0)} \sqcup \X^{(1)} \sqcup \dots \sqcup \X^{(N)}$
which is the disjoint union 
of $N$ finite  sets $\X^{(n)}$  for $n=1$, $2$, \dots, $N$.

A sample $x\in \mathcal N$ can be written
 \[ \bar x= (x^{(0)}, x^{(1)}, \dots , x^{(N)})\]
where $x^{(n) } \in 2^{\X^{(n)}}$.
Fix an integer $p$ and the following functions:
\begin{align} 
\phi_k:&\X^{(0)}\rightarrow \Real,\notag\\
\label{eqn:W-def}
W_n:&\X^{(n)}\times \X^{(n+1)}\rightarrow \Real\text{, and}\\
\psi_k:&\X^{(N)}\rightarrow \Real\notag
\end{align}
for $n=0$, \dots, $N-1$ and $k=1$, \dots, $p$.
The measure we are interested in has the form 
\begin{multline}
\label{eqn:measure}
\Probs[\{\bar x\}]=Z^{-1} \det [\phi_k(x_l^{(0)})]_{k,l=1}^p
\det [W_0(x_k^{(0)}, x_l^{(1)})]_{k,l=1}^p
\times \\\cdots \times 
\det [W_{N-1}(x_k^{(N-1)}, x_l^{(N)})]_{k,l=1}^p
\det [\psi_k(x_l^{(N)})]_{k,l=1}^p
\end{multline}
if $|x^{(1)}|=|x^{(2)}|=\dots =|x^{(N)}|=p$ 
and $\Prob^*[\{\bar x\}]=0$ otherwise. 
Here $x^{(n)}_l$ is the $l$:th element of the set $x^{(n)}$
for $n=1$, \dots, $N$ and $l=1$, \dots, $p$. 
For that to be well defined one needs a total ordering on $\X^{(n)}$,
but which one we use is not important. 
Note that the measure $\Probs:2^{2^\N}\rightarrow \Real$ is
defined by (\ref{eqn:measure}) for singleton sets and 
by additivity for all other sets in $2^{2^\N}$.
Measures of this form turn up everywhere in random matrix theory 
and related combinatorial models.

The idea now is to extend the space on which the 
point process lives in such a way that the new bigger
process admits an $L$-ensemble representation. 
Let $\bar \N=\{1, 2, \dots, p \}$ and let $\X=\bar\N\sqcup \N$.
The measure $\Probs$ can then be expressed as
\begin{equation*}
\Probs[\{\bar x\}] = 
\frac{\Prob[\{\bar x \sqcup \bar \N\}]}{\Prob[E_\X(\bar \N)]}
\end{equation*}
where $\Prob$ is the measure defined by (\ref{eqn:L-measure})
where 
\begin{equation}
\label{eqn:l-matrix}
L=\begin{bmatrix}
0 & \Phi & 0&0&\dots &0&0\\
0 & 0& -W_0 &0&\dots &0&0\\
0 & 0& 0 & -W_1 &\dots &0&0\\
\vdots \\
0 & 0 & 0&0&\dots &0&-W_{N-1}\\
\Psi & 0 & 0&0&\dots &0&0
\end{bmatrix}.
\end{equation}
Here $\Phi$, $W_1$, \dots, $W_{N-1}$ and $\Psi$
are certain blocks and $0$ means the
zero matrix of appropriate dimension.
The minus signs are for convenience later.
Recall that the $L$-matrix should be of size
$|\X| \times |\X|$ and that it's rows and columns are 
indexed by elements of 
$\X=\bar\N\sqcup\X^{(1)} \sqcup \X^{(2)} \sqcup \dots \sqcup \X^{(N)}$.
The determinants of the various blocks in~\eqref{eqn:l-matrix}
will be exactly the determinants that occur in~\eqref{eqn:measure}.

Here, $\Phi$ and $\Psi$ are matrices of dimension 
$p\times |\X^{(0)}|$ and $|\X^{(N)}|\times p$ respectively defined by 
\begin{align*}
[\Phi]_{n,x} &= \phi_n(x) && \text{for $n\in \bar\N$ and $x\in\X^{(0)}$,}\\
[\Psi]_{x,n}&= \psi_n(x) && \text{for  $x\in\X^{(N)}$ and $n\in\bar \N$.}
\end{align*}
The matrices $W_n$ for $n=1$, \dots, $N$ are 
of size $|\X^{(n)}|\times |\X^{(n+1)}|$ and defined by (\ref{eqn:W-def}).
Let 
\begin{equation}
\label{eq:W}
W_{[n,m)}= 
\begin{cases}
W_{n} \dots W_{m-1}, &n<m,\\
0, &n\geq m.
\end{cases}
\end{equation}
\begin{theorem}[Eynard-Mehta Theorem]
\label{thm:EynardMehta}
Assume 
\begin{equation}
\label{eqn:m-matrix}
M:=\Phi W_0 \cdots W_{N-1} \Psi
\end{equation}
is invertible. 
Then there exists a  correlation kernel $K^*$
for the  measure $\Prob^*$, that is,
\begin{equation}
\Probs[E_\N (X)] = \det K^*_X.
\end{equation}
This matrix can be written in block form 
\begin{equation*}
K^*=
\begin{bmatrix}
K^*_{0,0} &\dots & K^*_{0,N} \\
\vdots \\
K^*_{N,0} &\dots & K^*_{N,N} \\
\end{bmatrix}
\end{equation*}
where 
\begin{equation}
\label{eqn:EM-correlation-function}
K^*_{n,m} = W_{[n,N)} \Psi M^{-1} \Phi W_{[0, m)} - W_{[n,m)}.
\end{equation}
\end{theorem}
Note that the block $K^*_{n,m}$ is of size $|\X^{(n)}|\times |\X^{(m)}|$.
\begin{proof}
We shall use the following matrix identity. 
\begin{equation}
\label{eqn:matrix-inverse}
\begin{bmatrix}
A & B \\
C & D
\end{bmatrix}^{-1}
=
\begin{bmatrix}
-M\inv & M\inv B D\inv  \\
D\inv C M\inv  & D\inv - D\inv C M\inv B D\inv 
\end{bmatrix}
\end{equation}
Here, $A$ and $B$ must be square blocks and $M=BD\inv C - A$. 
This is easy to verify by explicit computation. 

Now according to Theorem \ref{thm:projected-kernel} we need to
invert $1_{(\N)} + L$ which can be decomposed as the left hand side of 
(\ref{eqn:matrix-inverse}) with $A=0$, $B=[ \Phi, 0, 0, \dots]$,
\begin{equation}
\label{eqn:dinv}
D\inv
=\begin{bmatrix}
\1& -W_0 &0&\dots &0\\
0& \1 & -W_1 &\dots &0\\
0& 0 & \1 &\dots &0\\
\vdots &&& \ddots\\
0 & 0&0&\dots &\1
\end{bmatrix}\inv 
=\begin{bmatrix}
\1& W_{[0,1)}  &W_{[0,2)}&\dots &W_{[0,N)}\\
0 & \1& W_{[1,2)} &\dots &W_{[1,N)}\\
0 & 0 & \1& \dots &W_{[2,N)}\\
\vdots &&& \ddots \\
0 & 0&0&\dots &\1
\end{bmatrix}
\end{equation}
and $C$ appropriately chosen.

Thus 
\begin{align*}
BD\inv& = [\Phi \quad \Phi W_{[0,1)} \quad \dots \quad \Phi W_{[0,N)} ],
&
D\inv C&= \begin{bmatrix}
 W_{[0,N)}\Psi\\
 W_{[1,N)}\Psi\\
\vdots\\
 \Psi
\end{bmatrix},
\end{align*}
and $M$ is as given by (\ref{eqn:m-matrix}).
Applying  (\ref{eqn:matrix-inverse})  to the formula 
in Theorem \ref{thm:projected-kernel} gives
\[
K^*= \1 -  D\inv + D\inv C M\inv B D\inv.
\]
Inserting the various ingredients above into this 
formula proves the theorem.
\end{proof}

\section{ Tips and Tricks} 
\label{sec:tips}
To compute the kernel in the main Theorem, the following
additional ideas are needed. 
None of these are new but for the instruction of 
the reader they are summarised here.
\subsection{Continuous state space}
In the version of the Eynard-Mehta theorem above, Theorem \ref{thm:EynardMehta},
the state spaces $\X^{(n)}$, for $n=1$, \dots, $N$, are finite. 
In the literature a version with $\X^{(n)} = \Real$ for all $n$ 
is more common. We shall now expound on the relationship
between these two versions of the same  useful theorem. 

Again we are faced with analysing a measure on the form 
(\ref{eqn:measure}) but now \\
$\Probs[\{x^*\}] $ is the probability density of configuration 
$x^*$,\\
$\phi_n:\Real \rightarrow \Real$, for $n=1$, \dots, $p$,\\
$W_n : \Real\times \Real \rightarrow \Real$, for $n=1$, \dots, $N$,\\
$\psi_n:\Real \rightarrow \Real$, for $n=1$, \dots, $p$.

The state space is now 
\begin{equation}
\label{eqn:continuous-state-space}
\Real\sqcup \Real \sqcup \cdots \sqcup \Real = N \times \Real.
\end{equation}

Pick some discretisation $\M$ of the real line, i.e.\ some 
sequence $\M_1$, $\M_2$, \dots\ of $|M|$ real numbers.
Restricting the measure in (\ref{eqn:measure})---with 
state space given by (\ref{eqn:continuous-state-space})---to 
the state space 
\begin{equation}
\M\sqcup \M \sqcup \cdots \sqcup \M = N \times \M
\end{equation}
gives a measure on a discrete set of exactly the type
to which Theorem~\ref{thm:EynardMehta} applies.
Now all the blocks in for example (\ref{eqn:dinv}) are
$|\M| \times |\M|$ and thus $D$ is a matrix of 
size $(N |\M| ) \times ( N |\M|)$. 
The idea is of course to take the limit $|\M|\rightarrow\infty$.

The correlation kernel (\ref{eqn:EM-correlation-function})
is computed through suitable matrix multiplications and
inversions. Consider for example the matrix multiplication 
of $W_1$ and $W_2$.
\begin{equation*}
\frac{1}{|\M|}
[W_1 W_2]_{x,z} = 
\frac{1}{|\M|}
\sum_{ y \in \M } [W_1 ]_{x,y}[ W_2]_{y,z} 
\rightarrow \int_\Real W_1(x,y) W_2(y,z)\,dy, \quad |\M| \longrightarrow \infty
\end{equation*}
The constant $|\M|\inv$ can be absorbed into the normalisation 
constant $Z$. 
Thus we see that all the matrix multiplications in the
expression (\ref{eqn:EM-correlation-function}) turn into
convolutions of the corresponding functions in 
the continuous setting. 
In later sections we will blur the line 
between discrete and continuous by sometimes using
the notation of matrix multiplication for convolutions.
It is understood that one needs to check convergence and integrability when 
one goes from the discrete to the continuous. 
That offers no problem in our examples so nothing further will be said on
that score.

\subsection{Unequal number of particles on each level.}
We shall illustrate this by the example of the GUE Minor process. 
Consider a GUE random matrix $M$ of size $N\times N$. 
Let the eigenvalues of the $n\times n$ minor, that is 
$[M_{ij}]_{i,j=1}^n$, be denoted $\lambda^{(n)}_1 > \cdots  > \lambda^{(n)}_n$.
Then these vectors $\lambda^{(1)}$, \dots, $\lambda^{(N)}$
can be seen as random variables.

For this process, the probability measure for all the variables 
$\bar \lambda = ( \lambda^{(1)}, \dots, \lambda^{(N)})$ is absolutely 
continuous with respect to the Lebesgue measure,
thus it has a \emph{probability density function} (p.d.f.).
It can be written \cite{Ba,JoNo06,FoNa} as 
\begin{equation}
p(\bar\lambda)= \frac{1}{C} \,
\1\{\lambda^{(1)} \prec \lambda^{(2)}\}
\cdots
\1\{\lambda^{(N-1)} \prec \lambda^{(N)}\} \,
\Delta(\lambda^{(N)}) 
\prod_{n=1}^N e^{-(\lambda^{(N)}_n)^2}.
\end{equation}
Here, $\lambda^{(n)} \prec \lambda^{(n+1)}$ means
$\lambda^{(n)}$ and $ \lambda^{(n+1)}$ interlace, 
i.e.\ $\lambda^{(n+1)}_1 \leq \lambda^{(n)}_1\leq \lambda^{(n+1)}_2\leq \cdots\leq \lambda^{(n+1)}_{n+1}$. 
We use the increasingly common notation that $\Delta$ denotes the Vandermonde 
determinant.
It turns out that it is practical to introduce fictitious (or virtual) 
variables 
$\lambda^{(0)}_1=\lambda^{(1)}_2=\lambda^{(2)}_3=\cdots=\lambda^{(N-1)}_N=-\infty$. 
Then the interlacing condition can be written in terms of 
determinants~\cite{Wa}
using the Heaviside function $H(x)=\1\{x\geq 0\}$ and the above p.d.f. becomes
\begin{multline}
\label{eqn:gue-measure}
p(\bar\lambda)= \frac{1}{C} 
\det[H(\lambda^{(1)}_i- \lambda^{(0)}_j )]_{i,j=1}^1 ) 
\det[H(\lambda^{(2)}_i- \lambda^{(1)}_j) ]_{i,j=1}^2 ) 
\cdots\\
\cdots\det[H(\lambda^{(N)}_i- \lambda^{(N-1)}_j )] )_{i,j=1}^N
\Delta(\lambda^{(N)}) 
\prod_{n=1}^N  e^{-(\lambda^{(N)}_n)^2}.
\end{multline}

The reader must agree that this vaguely resembles 
(\ref{eqn:measure}) except that the dimension of the matrices 
change. The first is a $1\times 1$ determinant and the 
last of size $N\times N$ for example.
Notice too that the last column of all these matrices 
is identically one because of our choice of fictitious particles above. 

The way to deal with this, first discovered in \cite{BoFePrSa}, is to 
form an $L$-matrix similar to the one in (\ref{eqn:l-matrix})
but which looks like this.
\begin{equation*}
\label{eqn:l-matrix-GUEm}
L=\begin{bmatrix}
0 & \Phi & 0&0&\dots &0&0\\
E_0 & 0& -W_0 &0&\dots &0&0\\
E_1 & 0& 0 & -W_1 &\dots &0&0\\
\vdots \\
E_{N-1} & 0 & 0&0&\dots &0&-W_{N-1}\\
\end{bmatrix}
\end{equation*}
Here, in the example of the GUE Minor process, 
 $[W_n]_{x,y}=H(y-x)$ for $n=0$, $1$, \dots, $N-1$, and $x$,~$y\in\M$. 
For the matrices in~\eqref{eqn:gue-measure} the last column---which 
is identically one---is moved out to the first column of blocks. 
Thus a sequence of $\M \times N$ matrices $E_0$, \dots, $E_{N-1}$
are produced such that, for $m=1$, \dots, $N$, 
\[
[E_n]_{x,m}=
\begin{cases}
1, & n+1=m,\\
0, & n+1\neq m.
\end{cases}
\]
By cranking this machinery it is possible to compute 
the correlation kernel for the GUE Minor process. 
This is done in \cite{BoFePrSa,FoNa}.

\subsection{Column operations on the kernel}
\label{sec:column-operations}
Recall the expression for the kernel in (\ref{eqn:EM-correlation-function}).
It is sometimes favourable to perform some column operations
on the matrices $W_{[n,N)} \Psi$ for $n=0$, \dots, $N-1$. 
Doing \emph{column operations} means multiplying from 
the right with an upper triangular or 
lower triangular matrix, say $R_m$, which 
must be invertible. 
Typically, but not necessarily,  it will have ones on the diagonal
 and a single off-diagonal entry. 
The kernel~\eqref{eqn:EM-correlation-function} then takes the form 
\begin{equation*}
K^*_{n,m} = W_{[n,N)} \Psi R_n (MR_n)\inv \Phi W_{[0, m)} - W_{[n,m)}
\end{equation*}
for $n$, $m\in\{0, \dots, N-1\}$.

\section{Markovness along space-like paths}
\label{sec:markov}
It is a sad fact of life that no transition density for
the Warren process is known explicitly. We do however know
certain marginals. To fix notation let 
$\bar x = (x^{(1)}, \dots, x^{(n)})$ where $x^{(k)}\in\mathbb{R}^k$.
Let $I^{(n)}(\bar x)=1$ if the interlacing $x^{(n)}\succ  \cdots \succ x^{(1)}$ holds
and $0$ otherwise. 
Let $P^{(n)} _t $ be the transition density for Warren's process which 
we know exists since it is a well defined stochastic process. 
In this notation, \cite[Proposition 6]{Wa} can be restated as follows.
\begin{proposition}
\label{prop:wa1}
For fixed $n$, $x^{(n)} $ and $\bar y=(y^{(1)}, \dots, y^{(n)})$, 
\begin{equation}
\int \frac{I^{(n)}(\bar x)}{\Delta(x^{(n)})} P^{(n)}_t(\bar x, \bar y) 
\, dx^{(1)}\cdots dx^{(n-1)} = 
\frac{I^{(n)}(\bar y)}{\Delta(y^{(n)})} p^{(n)}_t(x^{(n)}, y^{(n)}) 
\end{equation}
where 
\begin{equation}
p^{(n)}_t(x,y) = \frac{\Delta(y)}{\Delta(x)} 
\det\left[ e^{-(x_i-y_j)^2/t} \right]_{i,j=1}^n.
\end{equation}
\end{proposition}

Also, by the characterization after (30) in Warren's paper,
it is clear that 
\begin{proposition}
\label{prop:wa2}
For fixed $n$, $x^{(1)}$, \dots, $x^{(n-1)}$, $y^{(1)}$, \dots, $y^{(n-1)}$,
\begin{equation}
\int  P^{(n)}_t(x^{(1)}, \dots, x^{(n)}; y^{(1)}, \dots, y^{(n)}) 
\,  dy^{(n)} = P^{(n-1)}_t(x^{(1)}, \dots, x^{(n-1)}; y^{(1)}, \dots, y^{(n-1)}).
\end{equation}
\end{proposition}

We shall compute the eigenvalue measure along a particular space-like
path and the reader will see how to generalise this. 
Suppose we want to look at the path $(n, t_1)$, $(n, t_2)$, 
$(n-1, t_2)$, $(n-2, t_2)$, $(n-2, t_3)$ for some fixed $n$ and $0<t_1<t_2<t_3$.
The density of the event that the Warren process, started 
at the origin, takes values $\bar x$, $\bar y$ and $\bar z$ respectively
at times $t_1$, $t_2$ and $t_3$ respectively is by \cite{Wa}
\begin{equation}
\Delta^2(x^{(n)}) e^{-\sum_{i=1}^n  (x^{(n)}_i)^2}
\frac{I^{(n)} (\bar x)}{\Delta(x^{(n)})}
 P^{(n)}_t(\bar x, \bar y)  P^{(n)}_{t'}(\bar y, \bar z) 
\end{equation}
where $t:=t_2-t_1$ and $t':=t_3-t_2$.
To find the distribution on the afforementioned path we need
to integrate out  $x^{(1)}$, \dots, $x^{(n-1)}$,  $y^{(1)}$, \dots, $y^{(n-3)}$,
 $z^{(n-1)}$ and $z^{(n)}$. 

We start by integrating out the unwanted $x$-variables which can be done
by applying Proposition~\ref{prop:wa1} which gives
\begin{equation}
\label{eqn:markov2}
\Delta^2(x^{(n)}) e^{-\sum (x^{(n)})^2}
 p^{(n)}_t(x^{(n)}, y^{(n)})
\frac{I^{(n)}(\bar y)}{\Delta(y^{(n)})}
  P^{(n)}_{t'}(\bar y, \bar z) 
\end{equation}

Observing that 
\[I^{(n)} (\bar y) = \mathbbm{1} \{y^{(n)} \succ y^{(n-1)} \succ y^{(n-2)}\}
I^{(n-2)} (y^{(1)}, \dots, y^{(n-2)})\] and applying Proposition~\ref{prop:wa2}
twice we see that 
\begin{multline*}
\int I(\bar y)  P^{(n)}_{t'}(\bar y, \bar z) \, dz^{(n)}dz^{(n-1)} \\
=
\mathbbm{1} \{y^{(n)} \succ y^{(n-1)} \succ y^{(n-2)}\}
 P^{(n-2)}_{t'}(y^{(1)},\dots,y^{(n-2)}; z^{(1)}, \dots, z^{(n-2)}) 
\end{multline*}

We insert that into~\eqref{eqn:markov2} integrated and then apply 
Proposition~\ref{prop:wa1}
to integrate out the unwanted $y$-variables to get
\begin{multline}
\label{eqn:markov3}
\Delta(x^{(n)}) e^{-\sum (x^{(n)})^2}
\frac{\Delta(x^{(n)})}{\Delta(y^{(n)})}
 p^{(n)}_t(x^{(n)}, y^{(n)})
\mathbbm{1} \{y^{(n)} \succ y^{(n-1)} \succ y^{(n-2)}\}
\times \\
\times
\frac{\Delta(y^{(n-2)})}{\Delta(z^{(n-2)})}
 p^{(n-2)}_{t'}(y^{(n-2)}, z^{(n-2)}) 
\frac{I^{(n-2)} (z^{(1)}, \dots, z^{(n-2)})}{\Delta(z^{(n-2)})}
\end{multline}

It is well known that 
\[\mathbbm{1} \{y^{(n)} \succ y^{(n-1)}  = 
\det[\mathbbm{1} \{y^{(n)}_i >  y^{(n-1)}_j\} ] _{i,j=1}^n\]
if you adopt the convention that $y^{(n-1)}_n = -\infty$.
Then \eqref{eqn:markov3} can be written as a 
product of determinants.
\begin{multline}
\label{eqn:markov4}
\Delta(x^{(n)}) e^{-\sum (x^{(n)})^2}
\det\left[ e^{-(x^{(n)}_i-y^{(n)}_j)^2/t} \right]_{i,j=1}^n
\det[\mathbbm{1} \{y^{(n)}_i >  y^{(n-1)}_j\} ] _{i,j=1}^n
\times \\
\times
\det[\mathbbm{1} \{y^{(n-1)}_i >  y^{(n-2)}_j\} ] _{i,j=1}^{n-1}
\det\left[ e^{-(y^{(n-2)}_i-z^{(n-2)}_j)^2/t'} \right]_{i,j=1}^{n-2} 
\frac{I^{(n-2)} (z^{(1)}, \dots, z^{(n-2)})}{\Delta(z^{(n-2)})}
\end{multline}
The same idea can be applied to any other space-like path. 

An argument for the corresponding statement for
the Dyson Brownian minor process is given in~\cite[Section 4]{FeFr10}. \comment{
$ $Id: kernel.tex,v 1.71 2012/02/13 15:32:58 enord Exp $ $
}

\section{Definitions and computations}
\label{sec:definitions}
Given the theory presented in the last two sections, 
computing the kernel is nothing but a long tedious computation.
It was hard to write, hopefully it isn't too hard to read. 
The computation for the Warren process and the 
Dyson BM process can be done at the same
time with judicious choice of notation.
Let 
\begin{equation}
\label{eqn:heaviside}
H^n(x):=
\begin{cases}
(n-1)!\inv x^{n-1} \1\{x\geq 0\},  & \text{$n=1$, 2, \dots},\\
\delta(x), & n = 0.
\end{cases}
\end{equation}
be the $n$th anti-derivative of the Dirac delta function
and $H:= H^1$ be the Heaviside function.

\subsection{Dyson BM}
\label{sec:dyson}
Define the normalised Hermite polynomials,
\begin{equation}
h_n^*(x)= [\sqrt{\pi} n! 2^n]^{-1/2}   (-1)^{n}(w^*(x))\inv  D^n w^*(x)
\end{equation}
which are orthonormal with respect to the weight
\begin{equation}
w^*(x)=e^{-x^2}.
\end{equation}

The transition density of the Ornstein-Uhlenbeck process
is the well known expression
\begin{equation}
\label{eqn:pstar-def}
p_t^* (x,y)=\frac{\exp(\frac{-(y-q_t^*x)^2}{1-q_t^{*,2}})}{\sqrt{\pi(1-q_t^{*,2}) }}
\end{equation}
where 
\[
q_t^*=e^{-t}.
\]
While we are at it, define $r_t^*:=q_t^*$ and $\sigma^*(t)=1/\sqrt{2}$
for $t>0$, and set $q_t^{*,n}=(q_t^{*})^n$.

\subsection{Warren process}
\label{sec:warren}
For $t>0$, let
\begin{equation}
h_n^{(t)}(x)= [\sqrt{\pi} n! 2^nt^{-n}]^{-1/2}   
(-1)^{n}(w^{(t)}(x))\inv  D^n w^{(t)}(x).
\end{equation}
These are orthogonal with respect to the weight
\begin{equation}
w^{(t)}(x)=e^{-x^2/t}
\end{equation} 
and related to the Hermite polynomials above by 
$h^{(t)} (x)  = h^*(x/\sqrt t)$.

The transition density of Brownian motion with variance
$t/2$  is 
\begin{equation}
\label{eqn:p-def}
p_t (x,y)=\frac{1}{\sqrt{\pi t} }e^{-(y-x)^2/t}.
\end{equation}
Define 
\begin{equation}
q^{(t)}_s=\sqrt{ \frac t{s+t} }
\end{equation}
We are going to set $r_t \equiv 1$ 
and $\sigma(t)=\sqrt{t/2}$.

We need the following convolutions in our computations later.
\begin{lemma}
For $n=0$, $1$, $2$, \dots, and $t>0$,
\begin{align}
\label{eqn:time-step}
\int_\Real h^{(t)}_n(x) w^{(t)}(x) p_s(x,y)\,dx&=
q_s^{(t),n} h^{(t+s)}_n(y)w^{(t+s)}(y),\\
\label{eqn:space-step}
\int_\Real h_{n+1}^{(t)}(x) w^{(t)}(x) H(x-y) \, dx&
=\sigma(t)(n+1)^{-1/2} h_{n}^{(t)}(y) w^{(t)}(y),\\
\label{eqn:space-step-back}
\int_{\Real}  H^{n}(x-y)   H(y-z) \, dy&
=  H^{n+1}(x-z)\\
\label{eqn:time-semigroup}
\int_{\Real} p_t(x,y) p_s(y,z) \, dy &= 
p_{t+s}(x,z)\\
\label{eqn:time-step-back}
\int_\Real p_t(x,y)H^n(y-z)\,dy&= 
r_t^{n}  \int_\Real H^n(x-y)p_t(y,z)\,dy\\
\label{eqn:orthogonality}
\int_\Real  h^{(t)}_n(x)h^{(t)}_m(x) w^{(t)}(x)&=
\sqrt2 \sigma(t) \delta_{nm}\\
\label{eqn:space-shift}
p_t (x + y, z) &= p_t(x, z - r_t y), \\
\label{eqn:q-composition}
q^{(t_1)}_{t_2-t_1} q^{(t_2)}_{t_3-t_2} &= q^{(t_1)}_{t_3-t_1} \\
\label{eqn:r-composition}
r_t r_s &= r_{t+s} \\
\label{eqn:qr-composition}
\frac{q^{(s)} _{t-s} }{\sigma(s)} &= \frac{r_{t-s}}{\sigma(t)}
\end{align}
All of this is also true for the stared functions.
The coefficient of $x^n$ in $h_n^{(t)}$ is
\[
a_n:=\frac1{ \sigma^n(t)\sqrt{n!\sqrt{\pi}}}
\]
and that is true for $h_n^{*}$ with $\sigma$ replaced
by $\sigma^*$.
\end{lemma}

These are proved by explicit elementary computation. 
An important point here is that replacing $h$, $q$, $r$, 
$w$, $p$ and $\sigma$ with the stared versions 
these equations still hold. 
By this intelligent choice of notation we can do the computation concerning 
the Ornstein-Uhlenbeck and the Warren process at the same time. 
Furthermore,
\begin{lemma}
\label{thm:step-expansion}
\begin{multline}
\int_\Real H^n(x-y) p_{t-s} (y, z) \,dy = 
\sum_{k=0}^{n-1} 
\frac{h_k^{(s)}(x)(\sigma(t))^k q_{t-s}^{(s), k}}{r_{t-s}^n\sigma(s)\sqrt{2k!}\pi^{\frac14}}
\int_\Real  w^{(t)}(y) H^{n-k} (y-z) \, dy\\
+\frac{(\sigma(t))^{n} }{\sqrt2 \sigma(s) r^n_{t-s} }
\sum_{k=n}^{\infty} \sqrt\frac{(k-n)!}{k!} q_{t-s}^{(s), k} 
 h_k^{(s)}(x)  h_{k-n}^{(t)} (z) w^{(t)} (z) .
\end{multline}
\end{lemma}
\comment{Checked again 100521.}
\begin{proof}
By orthogonality 
\begin{equation}
f(x, z) = \int H^n(x-y) p_{t-s} (y, z) \,dy 
\end{equation}
can be written as 
\begin{equation}
f(x, z) = \sum_{k=0}^\infty c_k(z) h_k^{(s)}(x) 
\end{equation}
for suitable coefficients 
\begin{align}
(\sqrt2 \sigma(s) ) c_k(z) &= \int_\Real f(x,z) h_k^{(s)}(x) w^{(s)}(x) \,dx=\\
&= \iint_{\Real\times \Real} 
 H^n(x-y) p_{t-s} (y, z) h_k^{(s)}(x) w^{(s)}(x) \, dx\,dy.
\intertext{Let's start with the case $k\geq n$. 
Apply~\eqref{eqn:time-step-back}.}
&=\frac{1}{r_{t-s}^n} \iint _{\Real\times \Real} 
h_k^{(s)}(x) w^{(s)}(x)p_{t-s} (x, y)  H^n(y-z) \, dx dy
\intertext{Apply (\ref{eqn:time-step}). }
\label{eqn:coefficients}
&=\frac{q_{t-s}^{(s), k}}{r_{t-s}^n}  \int _{\Real} 
h_k^{(t)}(y) w^{(t)}(y)  H^n(y-z) \, dy
\intertext{
Apply~\eqref{eqn:space-step} and~\eqref{eqn:space-step-back} 
$n$ times to get }
&=\frac{q_{t-s}^{(s), k}}{r_{t-s}^n} (\sigma(t))^n \sqrt\frac{(k-n)!}{k!}
  h_{k-n}^{(t)}(z) w^{(t)}(z) 
\intertext{ Now suppose instead that $k<n$. Then
everything up to~(\ref{eqn:coefficients}) works the same. 
Apply (\ref{eqn:space-step}) $k$ times. }
&=\frac{q_{t-s}^{(s), k}}{r_{t-s}^n}\frac{(\sigma(t))^k}{\sqrt{k!}}
\int_\Real h_0^{(t)}(y) w^{(t)}(y) H^{n-k} (y-z) \, dy\\
\intertext{Remember that $h_0(y)\equiv \pi^{-\frac14}$.}
&=\frac{q_{t-s}^{(s), k}}{r_{t-s}^n}\frac{(\sigma(t))^k}{\sqrt{k!}\pi^{\frac14}}
\int_\Real  w^{(t)}(y) H^{n-k} (y-z) \, dy
\end{align}
\end{proof}

\subsection{Integral representations.}
\label{sec:hermite-representations}
With the normalisations above the classical integral
representations for the Hermite polynomials are
\begin{equation}
\label{eqn:hermite-repressentation1}
h^{(t)}_n (x) = \frac{\pi^{\frac14} e^{x^2/t} (2t)^{n/2} }{i\pi \sqrt{n!}}
\int_\Gamma v^n e^{v^2-2vx/\sqrt{t} } \,  dv
\end{equation}
and
\begin{equation}
\label{eqn:hermite-repressentation2}
h^{(t)}_n (x) = \frac{(t/2)^{n/2} \sqrt{n!}}{\pi^{\frac14}} 
\frac{1}{2\pi i } 
\int_\gamma u^{-n-1} e^{-u^2 + 2ux/\sqrt{t} } \,  du
\end{equation}
with contours of integration as in Figure \ref{fig:contours}.
The starred Hermite polynomials are given,  
for $n=0$, 1, \dots, by  $h_n^{*} \equiv  h^{(1)}_n$. 

As a sort of generalisation of~\eqref{eqn:hermite-repressentation1}
it can be shown~\cite{No09}
that  for $n=1$, 2, \dots,
\begin{equation}
\frac{1}{\pi i } 
\int_\Gamma v^{-n} e^{v^2-2vx/\sqrt{t} } \, dv
= \frac{ 2^n}{\sqrt{\pi} t^{n/2} }
 \int_\Real H^n(y-x) e^{-y^2/t} \, dy.
\end{equation}

\begin{figure}
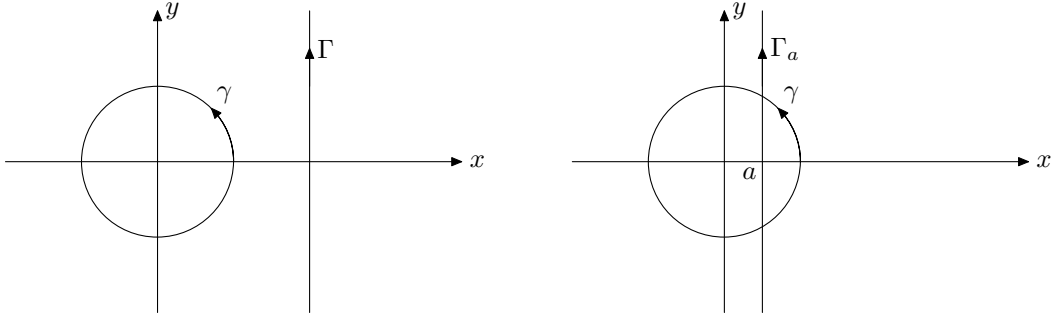

\includegraphics{fig0.mps} \hspace{2em}
\includegraphics{fig1.mps} 
\caption{Contours of integration.}
\label{fig:contours}
\end{figure}

\section{The Kernel}
\label{sec:kernel}
\begin{figure}
\includegraphics[width=15cm]{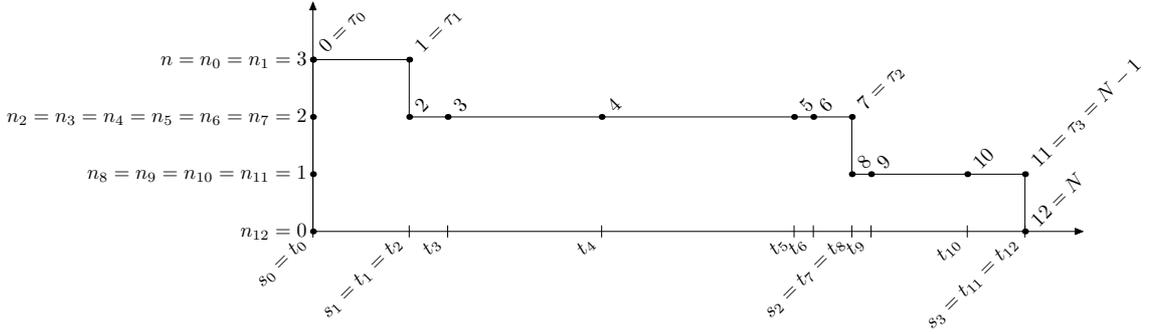}
\caption{Times and levels must follow a space-like path.
This means that this curve must not take steps upward. }
\label{fig:notation}
\end{figure}

The setup now is the following. 
Pick $N$ times and levels $(t_0,n_0)$, \dots, $(t_{N-1},n_{N-1})$ following
a space-like path. 
This means that $0<t_0\leq t_1\leq \dots \leq t_{N-1}$
and $n_0\geq n_1\geq \dots \geq n_{N-1}$.
Without loss of generality we can take $n_{N-1}=1$
and $n_m-n_{m+1}\in \{0,1\}$ for all $m=0$, \dots, $N-2$.
That is, we end at level 1 and only drop one level at a time.
For the sake of notation let $n_N=0$ and $t_N = t_{N-1}$.
Denote by $x^{(k)}=( x_1^{(k)}\geq x_2^{(k)}\geq \dots \geq x_{n_k}^{(k)})$
the $n_k$ eigenvalues at time $t_k$ and level $n_k$.
We shall say that $m\in \Space $ if 
the $m$th step is a down step, i.e. $n_m = n_{m+1} + 1$.
For $k=1$, \dots, $n_0$, let $\tau_k$ and $s_k$ be the 
position and time, respectively, of the $k$th down step. 
Thus $\tau_k\in \Space$ by definition and $n_{\tau_k}=n+1-k$. 
The time of the $k$th down step is $s_k=t_{\tau_k} = s_{n_0+1-n_{\tau_k}}$.
Also let $\tau_0 = 0$ and $s_0= t_0$.
If the $m$th step is a time step, i.e. $n_m = n_{m+1}$,
then we shall write $m\in\Time$.
That's a lot of notation, hopefully Figure \ref{fig:notation}
should make this clearer.

Let 
\begin{align}
\phi_k(x)&= h^{(t_0)}_{k-1} (x) w^{(t_0)}(x) & &\text{for $k=1,\dots, n_0$}\\
W_m(x,y)&= \begin{cases}
p_{t_{m+1}-t_{m}}(x,y) & m\in\Time\\
H(x-y) & m\in\Space
\end{cases}
&&\text{for $m=0$, \dots, $N-1$.}
\end{align}

A full configuration of eigenvalues
is the $N$-tuple 
\[
\bar x = (x^{(0)}, x^{(1)}, \dots, x^{(N-1)})
\]
where $x^{(m)} \in \Real^{n_m}$ for $m=0$, \dots, $N-1$.  
We adopt the notation that $x^{(m+1)}_{n_m+1}=u$ for some
large negative real number $u$ if $m\in\Space$.  
In particular $x^{N}_1=u$ since, by definition $N-1\in\Space$. 
Those are the positions where we step down a level, 
that is, we lose an eigenvalue. One way to think of this is
that that particle jumps away to some position  $u$ 
which is close to $-\infty$. 
The weight or probability density of configuration $\bar x$ is then, 
by the Markovness along space-like paths discussed in 
Section~\ref{sec:markov}, given by the following product
\begin{equation}
\label{eq:weight}
p(\bar x ) = Z\inv \det [ \phi_k(x^{(0)}_l) ]_{k,l=1}^{n_0}
\prod_{m=0}^{N-1}\det [ W_m(x^{(m)}_k,x^{(m+1)}_l )]_{k,l=1}^{n_m},
\end{equation}
where we adopt the notation that $x^{(m+1)}_{n_m+1}=u$ for some
large negative real number $u$ if ${n_m-n_{m+1}}=1$, for $m=1$, \dots, 
$N$. Also $x^{N}_1=u$. 

Let the magic begin.
Define the block matrix 
\begin{equation}
L=\begin{bmatrix}
0 & \Phi & 0 & 0 & \dots & 0 & 0\\
F_0 & 0 & -W_0 & 0 & \dots & 0 & 0\\
F_1 & 0 & 0 & -W_1 & \dots & 0 & 0\\
\vdots\\
F_{N-2} & 0 & 0 & 0 & \dots & -W_{N-2}&0\\
F_{N-1} & 0 & 0 & 0 & \dots & 0 & -W_{N-1}\\
F_{N} & 0 & 0 & 0 & \dots & 0 & 0
\end{bmatrix}.
\end{equation}
The $F_{N}$ block in the above 
block matrix has zero rows but let's keep it notation.
The $(W_n)_{n=0}^{N-1}$ are defined above.
For $l=1$, \dots, $n_0$,
\begin{align}
\label{eqn:phi-zero-definition}
[\Phi]_{l,x} &= \phi_{n_0+1-l}(x) = h_{n_0-l}^{(t_0)}(x) w^{(t_0)}(x),\\
\label{eqn:f-definition}
[F_k]_{x,l}&=\begin{cases}
H(x-u), & k=\tau_{n_0-l+1},\\
0, & \text{otherwise.}
\end{cases}
\end{align}
The measure in~(\ref{eq:weight}), being similar to
that in~\eqref{eqn:measure},  can then be represented 
as in~(\ref{eqn:L-measure}) with the above $L$-matrix.
Introduce $W_{[k,l)}$ as in (\ref{eq:W}).
By the general theory of these $L$-ensembles, we
need to compute 
\begin{equation}
\label{eqn:linear-algebra-kernel}
K= \1 -  D\inv + D\inv C M\inv B D\inv
\end{equation}
for some invertible  matrix $R$. 
Here, $D$ is as in (\ref{eqn:dinv}), $M=BD\inv C$,
\begin{align}
\label{eqn:b}
B&=\begin{bmatrix}\Phi & 0 & 0 & \dots & 0\end{bmatrix},\\
C&=\begin{bmatrix}
F_0 \\
F_1 \\
\vdots \\
F_{N-1}
\end{bmatrix}.
\end{align}
The scheme things now is to analyse each of the different
components of~(\ref{eqn:linear-algebra-kernel}),
namely $BD\inv$, $D\inv C $, $(M)\inv$ and the upper
triangular matrix $1-D\inv$.

First of all, 
$$BD\inv  =
\begin{bmatrix}
  \Phi &
  \Phi W_{[0,1)}&
  \Phi W_{[0,2)}&
\dots &
  \Phi W_{[0,N)}
\end{bmatrix}
$$
which is an $n_0\times |\M|$ matrix.
We'll call the $k$th block of this $\Phi_k$. 
That is, $\Phi_0 := \Phi$ and 
\begin{equation}
  \label{eqn:phi-definition}
  \Phi_k:=\Phi W_{[0,k)},  
\end{equation}
for $k = 1$, \dots, $N$,  and an explicit expression for it will
be  given in Lemma \ref{thm:bd-inv}. 
Next, let's look at 
\begin{equation}
  \label{eqn:dinv-c}
D\inv C= 
\begin{bmatrix}
F_0 +  W_{[0,1)} F_1 +  W_{[0,2)}F_2 + W_{[0,3)}F_3 + \cdots + W_{[0,N)}F_N \\
 F_1 +  W_{[1,2)}F_2 + W_{[1,3)}F_3 + \cdots + W_{[1,N)}F_{N} \\
\vdots\\
F_{N}
\end{bmatrix}
=:
\begin{bmatrix}
  \bar\Psi_0 \\
  \bar\Psi_1 \\
  \vdots\\
  \bar\Psi_N
\end{bmatrix}
\end{equation}
where, for $k=0$, \dots, $N$,
\begin{equation}
  \label{eqn:psi-bar-definition}
  \bar\Psi_k = F_k + \sum_{j = k+1} ^N W_{[k, j)} F_j.  
\end{equation}
As mentioned we will have to do column operations on this,
which are represented by the $n_0\times n_0$ matrix $R_k$, 
see~\eqref{eqn:EM-correlation-function}. 
We will choose $R_k$ in such a way that
$MR_k$, for $k=0$, \dots, $N-1$, 
 is asymptotically the identity matrix as $u\rightarrow -\infty$.
Thus for $k=0$, \dots, $N$ let 
\begin{equation}
  \label{eqn:psi-definition}
  \Psi_k := \bar\Psi_k R_k   . 
\end{equation}
These will be explicitly computed in Lemmas~\ref{thm:psi-bar} 
and~\ref{thm:psi}.

Block $(k,k') $ of the kernel in~\eqref{eqn:linear-algebra-kernel}
can with this notation be written, 
remembering~\eqref{eqn:dinv}, as
\begin{align}
[K_{k,k'}]_{x,y} &= 
-W_{[k,k')} + 
\sum_{i,j = 1} ^{n_0}  
[\bar \Psi_k]_{x,i} [ M\inv ]_{i,j} [\Phi_{k'}]_{j,y} \\
\label{eqn:psi-phi-sum}
&= -W_{[k,k')} +  \sum_{i,j = 1} ^{n_0}  
[\Psi_k]_{x,i} [ (MR_k)\inv ]_{i,j} [\Phi_{k'}]_{j,y}. 
\end{align}
Note that, as we shall see in Lemma~\ref{thm:psi}, 
only the columns 1, \dots, $n_k$ of $\Psi_k$ are
non-zero. So we need only compute rows 
  1, \dots, $n_k$ of $(MR_k)\inv$. 
As it happens we never need to explicitly write down
what $R_k$ is but it can in principle be extracted from the 
proof of Lemma~\ref{thm:psi}. 

Recall from~\eqref{eqn:linear-algebra-kernel}
that $M = BD\inv C$, but what we really need
is $MR_k$  for $k=0$, \dots, $N-1$. Multiply~\eqref{eqn:b} 
with~\eqref{eqn:dinv-c} to compute
\begin{equation}
MR_k = BD\inv CR_k = \Phi F_0 R_k
 +  \sum_{j = 1} ^N \Phi W_{[0, j)} F_j R_k . 
\end{equation}  
For $k = 0 $ that specialises to $MR_0 = \Phi_0\Psi_0$.
For $k>0$ it will later turn out that we 
only need the first $n_k$ columns of $MR_k$. 
That allows us to remove those terms that only contribute
to columns $n_k + 1$ to $n_0$, which by~\eqref{eqn:f-definition}
is those that involve $F_0$, \dots, $F_{k-1}$. 
We shall denote by $\simeq$ the operation of 
removing the unnecessary columns.
\begin{align}
MR_k &\simeq
 \sum_{j = k} ^N \Phi W_{[0, j)} F_j R_k  \\
&= \Phi W_{[0,k)} ( F_k + \sum_{j = k+1} ^N  W_{[k, j)} F_j) R_k\\
&= \Phi_k \Psi_k
\end{align}  

Another way of saying that is 
\begin{equation}
\label{eqn:mr}
[MR_k] _{i,j} = \int_\Real [\Phi_k]_{i,x} [\Psi_k]_{x,j} dx.
\end{equation}
for $i$, $j=1$, \dots, $n_k$. 

\subsection{Computing $1-D\inv$.}
From  (\ref{eqn:dinv}) it is clear that this is a 
matrix of size $N |\M| \times N |\M|$.
Furthermore, the $(k, l)$ block of this  is identically $0$ 
if $l\leq k$ and otherwise $-W_{[k, l)}$. 
Recall from~\eqref{eq:W} the definition of $W_{[k, l)}$.
For notation introduce a function $S:\{0,\dots,N-1\} \rightarrow \Real$
defined by 
\begin{equation}
  S(k): = \prod_{l=1}^{n_k} r_{s_{n_0 + 1 - l} - s_0} .
\end{equation}
\begin{lemma}
\label{thm:forward-transition}
\begin{equation}
\label{eqn:forward-transition}
[W_{[k, k')}]_{x,z} =
(r_{t_{k}-t_0})^{n_{k'}-n_{k}}
\frac{S(k)}{S(k')}
\int_\Real H^{n_k-n_{k'}} (x-y) p_{t_{k'} - t_k} (y, z) \,dy
\end{equation}
for $k<k'$ and $[W_{[k, k')}]_{x,z}=0 $ otherwise.
\end{lemma}
\comment{Checked again 100521.}
Recall the definition of $H$ in~\eqref{eqn:heaviside}.
\begin{proof}
This is shown by induction and $k'=k+1$ is the basic case. 
Let's say $k\in \Space$. 
(Recall from the first paragraph of Section~\ref{sec:kernel}
what that means.) 
Then $s_{n_0+1-n_k} = t_k$ and $n_k = n_{k'}+1$
so the above expression reduces to $[W_{[k, k+1)}]_{x,z} = H(x-z)$ which 
is correct. 
Otherwise $k\in \Time$. Then all the $ S $ and $r$ factors 
in~(\ref{eqn:forward-transition}) equals one. 
The integral evaluates to $[W_{[k, k+1)}]_{x,z} =p_{t_{k+1} - t_k}(x,z)$ 
which is correct.

Of course
\begin{equation}
[W_{[k, k'+1)}]_{x,z} = 
\int_\Real  [W_{[k,k')} ]_{x,y} [W_{k'}] _{y,z} \, dy.
\end{equation}
If $k'\in\Space$ then apply first~\eqref{eqn:time-step-back}
which pops out an $r_{t_{k'}-t_k}= r_{t_{k'} - t_0} / r_{t_{k} - t_0} $ and then~\eqref{eqn:space-step-back}.
But $k'\in\Space$ implies that $t_{k'} = s_{n_0+1-n_{k'}} = s_{n_0+1-(n_{k'+1} +1)}$.
On the other hand if $k'\in\Time$, 
just a single application of~\eqref{eqn:time-semigroup} 
completes the induction. 
\end{proof}

\begin{lemma}
\label{thm:forward-transition-expanded}
\begin{multline}
\label{eqn:forward-transition-expanded}
[W_{[k, k')}]_{x,z} =
r^{n_{k'} - n_k} _{t_{k'} - t_0} \frac{S(k)}{S(k')} 
\sum_{l=n_{k'}+1}^{n_{k}} 
\frac{h_{n_k-l}^{(t_k)}(x)(\sigma(t_{k}))^{n_{k}-l -1} r_{t_{k'}-t_k}^{n_{k}-l }}{
 \sqrt{2(n_{k}-l )!}\pi^{\frac14}}
\int_\Real  w^{(t_{k'})}(y) H^{l-n_{k'}} (y-z) \, dy\\
+
r^{n_{k'} - n_k} _{t_{k'} - t_0} \frac{S(k)}{S(k')} 
\frac{(\sigma(t_{k'} ))^{n_k-n_{k'}} }{\sqrt2 \sigma(t_k) }
\sum_{l=-\infty}^{n_{k'}} \sqrt\frac{(n_{k'}-l)!}{(n_{k}-l)!}
 q_{t_{k'} -t_k}^{(t_k), n_{k}-l} 
 h_{n_{k}-l}^{(t_k)}(x)  h_{n_{k'}-l}^{(t_{k'})} (z) w^{(t_{k'})} (z) .
\end{multline}
for $k'>k$ and $[W_{[k, k')}]_{x,z}=0 $ otherwise.
\end{lemma}
\begin{proof}
Apply 
Lemma \ref{thm:step-expansion} 
and~\eqref{eqn:q-composition}--\eqref{eqn:qr-composition} to 
the formula in Lemma~\ref{thm:forward-transition}.
\end{proof}

\subsection{Computing  $BD\inv$.}
Recall from~(\ref{eqn:phi-definition}) the definition of 
$\Phi_k$, which is what we shall now compute. 
Multiply~\eqref{eqn:phi-zero-definition}
with~\eqref{eqn:forward-transition-expanded} 
and use the orthogonality~\eqref{eqn:orthogonality} to get

\begin{lemma}
\label{thm:bd-inv}
\begin{equation}
\label{eqn:phi-sum-one}
[\Phi_{k'}]_{l,z} =
r^{n_{k'} - n_0} _{t_{k'} - t_0} \frac{S(0)}{S(k')} 
(\sigma(t_{k'} ))^{n_0-n_{k'}}
\sqrt\frac{(n_{k'}-l)!}{(n_{0}-l)!}
 q_{t_{k'} -t_0}^{(t_0), n_{0}-l} 
 h_{n_{k'}-l}^{(t_{k'})} (z) w^{(t_{k'})} (z) 
\end{equation}
for $l\leq n_{k'}$ (equivalently $k'\geq \tau_{n+1-l}$) 
and 
\begin{equation}
\label{eqn:phi-sum-two}
[\Phi_{k'}]_{l,z} =
r^{n_{k'} - n_0} _{t_{k'} - t_0} \frac{S(0)}{S(k')} 
\frac{(\sigma(t_{0}))^{n_{0}-l } r_{t_{k'}-t_0}^{n_{0}-l }}{
 \sqrt{(n_{0}-l )!}\pi^{\frac14}}
\int_\Real  w^{(t_{k'})}(y) H^{l-n_{k'}} (y-z) \, dy
\end{equation}
for~$l\geq n_{k'}$.
\end{lemma}
\subsection{Computing $D\inv C$.}
We do the same at the other end. 
Recall the definition of $\Psi_k$ in~\eqref{eqn:psi-bar-definition}.
Explicit computations, similar to what we did for 
$\Phi_k$ in the preceding section, leads us to 
conjecture a general expression, which is 
proved by induction.
\begin{lemma}
\label{thm:psi-bar}
\begin{equation}
\label{eqn:psi-bar}
[\bar \Psi_k]_{x,l}=
\left(\prod_{j=l}^{n_k} r_{s_{n_{0}+1-j} - t_k} \right)\times
\int_\Real H^{n_k+1-l}(x-y) 
p_{s_{n_0+1-l}-t_k}(y,u)
\,dy
\end{equation}
for~$n_k\geq l$ (equivalently~$k\leq \tau_{n_0+1-l}$) 
and~$ [\bar\Psi_k]_{x,l} \equiv 0 $ otherwise.
\end{lemma}
\begin{proof}
To get the hang of it, consider column one of $\bar\Psi_{k} $ for all 
$k=0$, 1, \dots. 
Recall~\eqref{eqn:f-definition}. 
Of course $F_N$ is quite tame, it is identically zero
giving us $\bar\Psi_N=0$.
By definition $\tau_{n_0}=N-1$, so step number 
$N-1\in\Space$ and $\bar\Psi_{N-1} = F_{N-1}$ is non-zero.
Remember that $F_{N-1}$ is non-zero only in 
column one and it is the only one of $F_0$, \dots, $F_{N-1}$ that 
is non-zero in that column. 
So
$$
[\bar\Psi_k]_{x,1} = [W_{[k, N-1)} F_{N-1}]_{x,1}.
$$
Repeated applications of~\eqref{eqn:time-step-back}
and~\eqref{eqn:space-step-back} specialises this
to~\eqref{eqn:psi-bar} if you keep track of the constants. 

More generally, for $k=0$, \dots, $N$,  
recall from~\eqref{eqn:f-definition} that column $l$ of $F_k$ is 
only non-zero if $k = \tau_{n_0+1-l}$. 
Remember too from~\eqref{eqn:psi-bar-definition} that $F_k$
only occurs in $\bar\Psi_0$, \dots, $\bar\Psi_k$.  
That accounts for the fact that $[\bar\Psi_k]_{x,l}\equiv 0 $ 
for $k> \tau_{n_0+1-l}$  (equivalently $n_k<l$). 

The induction on $k$ is now done backwards. 
Fix $l\in \{ 1, \dots, n_0\}$. 
We shall show that the formula~\eqref{eqn:psi-bar} 
is true for this particular value of $l$. 
Specialising to  $\bar k = \tau_{n_0+1-l}$,
equivalently $t_{\bar k} = s_{n_0+1-l}$, 
gives us $[\bar\Psi_{\bar k} ]_{x,l} = r_0 H^1(x-u)=F_{\bar k}$. 
Assume now that the theorem gives the correct expression
for $\bar\Psi_{k+1}$.
 Since $k \neq \bar k+1$ it is clear that $F_{k+1} = 0$
in column $l$. Also, $W_k W_{[k+1, \bar k} =  W_{[k, \bar k} $
So 
\begin{equation}
  [\bar\Psi_{k}]_{x,l} = \int_\Real  [W_k]_{x,y} [\bar\Psi_{k+1} ]_{y, l} \, dy
\end{equation}
If $k\in\Time$ then applying first~\eqref{eqn:time-step-back} 
and then~\eqref{eqn:r-composition} completes 
the proof.  If $k\in\Space$ then~\eqref{eqn:space-step-back}
does the same. 
The calculation is in reality only
a question of keeping track of the coefficients. 
\end{proof}
We now exercise the right to do column operations.
Recall from~\eqref{eqn:psi-definition} the
definition of $\Psi_k$. 
$R_k$ is chosen judiciously to give a very simple
expression for $MR_k$. 

\begin{lemma}
\label{thm:psi}
There exists an $R_k$ such that for  $x\in \M$ and 
$u$ a large negative real number, 
\begin{multline}
[\Psi_k]_{x,l}=
h_{n_k-l}^{(t_k)}(x)
r^{n_0-n_k} _{t_k-t_0} 
q^{(t_0), l - n_0}_{t_k-t_0} 
\sqrt\frac{(n_0-l)!}{(n_k-l)!} 
\frac{S(k)}{S(0)} \frac{(\sigma (t_k))^{n_k-n_0} }{ \sqrt2\sigma(t_k)}
+O(e^{-u^2/2})
\end{multline}
 for $l = 1$, \dots, $n_k$, and 
$[\Psi_k]_{x,l}= 0$
 for $l = n_k+1$, \dots, $n_0$.
\end{lemma}

\begin{proof}
This bit is rather hairy actually.
Fix $k$ and let $\rho_l = s_{n_0+1-n_k+l} - t_k$.
Let 
\begin{align*}
f_l(x,u) &= C(l)  [\bar \Psi_k ]_{x, n_k-l} 
\intertext{
for $l\in \{0, \dots, n_k-1\}$ where $C(l)$ is the constant that gives }
f_l(x,u) &= \int_\Real H^{l+1} (x-y)  p_{\rho_l} (y, u) \, dy\\
&= \int_\Real H^{l+1} (x-y)p_{\rho_l} (y - r_{\rho_l}\inv  u, 0) \, dy\\
&= \int_\Real H^{l+1} (x-y- r_{\rho_l}\inv u)p_{\rho_l} (y , 0) \, dy.
\end{align*}
The second equality is due to~\eqref{eqn:space-shift} and
the third one is a change of variables $y\mapsto y+ r_{\rho_l}\inv u$.
Now what we need to do is to produce a sequence
of functions $g_0$, \dots, $g_{n_k-1}$ such that
for each $l$,
\begin{enumerate}
\item 
$g_l$ is a linear combination of $f_0$, \dots, $f_l$,
\item
$g_l(x,u)$ differs from  polynomial in $x$ of order $l$ 
by at most a term exponentially decreasing in $u$. 
That is, there is a polynomial $\bar h_l(x) $  of 
degree $l$ such that 
$$
g_l(x,u) = \bar h_l(x) + O(e^{-u^2/2} ).
$$ 
\end{enumerate}
Just to  see how it works  consider 
\begin{align}
  f_0(x,u) &= \int_\Real H(x-y-r_{\rho_0}\inv u) p_{\rho_0} (y, 0) \, dy \\
   &= \int   p_{\rho_0} (y, 0) \, dy + O( e^{-u^2/2} )
\end{align}
as $u\rightarrow -\infty$. So let $g_0 = f_0$
implying that $h_0 $ is a constant.

Now for some constant 
$$C = \left (\int_\Real  p_{\rho_1}(y,0)\,dy\right)\left( \int_\Real  p_{\rho_0}(y,0) \, dy\right)\inv$$ look at 
\begin{align*}
f_1(x,u) -  C r_{\rho_1}\inv u f_0(x,u)  =&
 \int (x-y-r_{\rho_1}\inv u )  H(x-y-r_{\rho_1}\inv u ) p_{\rho_1}(y,0)\\
&- C  r_{\rho_1}\inv u \int  H(x-y-r_{\rho_1}\inv u) p_{\rho_0}(y,0)\\
\intertext{(where $C$ is chosen so that, in the
limit of large negative $u$,  the integrands partially cancel out.)}
=& \int (x-y )  p_{\rho_1}(y,0)\,dy + O( e^{-u^2/2})\\
=& x \int p_{\rho_1}(y,0)\,dy 
-  \int y p_{\rho_1}(y,0)\,dy + O(e^{-u^2/2}).
\end{align*}
Thus set $g_1(x,u) = f_1 -  C r_{\rho_1}\inv u f_0 $ 
and $\bar h_1 $ accordingly. 

Now that we see it works for $l = 0$ and 1, let's do the induction.
Say the statement is true for $l$, then 
\begin{align*}
f_{l+1} ( x,u )  & = \int (x-y-r_{\rho_{l+1}}\inv u )^{l+1}   
H(x-y-r_{\rho_{l+1}} \inv u ) p_{\rho_{l+1}}(y,0)\, dy\\
&= \sum_{i = 0} ^{l+1} \binom{l+1}{i}\int  (x-y)^{l+1-i} (r_{\rho_{l+1}}\inv u )^i 
H(x-y-r_{\rho_{l+1}}\inv ) p_{\rho_{l+1}}(y,0)\,dy
\end{align*}
The first term is already  a polynomial in $x$ of degree $l+1$
 modulo an O-term. The other terms are polynomials in $x$
of lower degree which by the induction hypothesis 
 can be expressed in $g_0$, \dots, $g_l$.

What we have effectively done is that we have
taken various linear combinations, encoded in $R$, 
of columns of $\bar\Psi_k$  producing
$$
[\bar \Psi_k R ] _{x, l} = \text{polynomial of degree $l$} + \text{perturbation exponentially small  in $u$.}
$$
Now we will go ahead and do more column operations 
and  multiply the columns with suitable constants.
Let $R'$ be the matrix that encodes the column operations
and multiplications that turn the columns of 
$\bar \Psi_k R$ into Hermite polynomials with factors
as in the statement of the Lemma. 
Set $R_k = R R'$ and $\Psi_k  = \bar \Psi_k R_k$ and we are done. 
\end{proof}

\subsection{Computing $(MR_k)\inv$}
First observe that $M$ is essentially lower triangular. 
To see this, for $1\leq i<j\leq n_0$ consider
\begin{equation}
\label{eqn:m-lower-triangular}
[M]_{ij} = \int_\Real   [\Phi_0]_{i,x} [\bar \Psi_0]_{x,j} \, dx= 
\cdots = O(e^{-u^2 / 2})
\end{equation}
This is a computation. You insert the expressions for $\Phi_0$ and
$\Psi_0$, apply first~\eqref{eqn:space-step-back}, then 
then~\eqref{eqn:time-step} and lastly perform $n_0-j$ applications 
of~\eqref{eqn:space-step}. 
By~\eqref{eqn:orthogonality} the conclusion in~\eqref{eqn:m-lower-triangular}
holds.

Let $I_n$ be the identity matrix of size $n\times n$. 
In block form $M$ and $R_k$ are:
\begin{equation}
MR_k = \begin{bmatrix}
* & O(e^{-u^2 / 2})\\
* & *
\end{bmatrix}
\begin{bmatrix}
* & 0\\
0 & I_{n_0-n_k} 
\end{bmatrix}
=
\begin{bmatrix}
I_{n_k} + O(e^{-u^2 / 2}) & O(e^{-u^2 / 2})\\
* & *
\end{bmatrix}
\end{equation}
The top-right block of $MR_k$ is essentially $I_{n_k}$ by construction. 
To see that, recall~(\ref{eqn:mr}), 
insert Lemmas~\ref{thm:bd-inv} and~\ref{thm:psi}
and apply~\eqref{eqn:orthogonality}.
To invert this matrix use~\eqref{eqn:matrix-inverse}
with $M=-I_{n_k} + O(e^{-u^2 / 2})$ and $B= O(e^{-u^2 / 2})$ and 
we are only interested in the top two blocks. 
\begin{proposition}
\begin{equation}
[(MR_k)\inv]_{ij}
=\delta_{i,j} + O( e^{-u^2/2})
\end{equation}
for $i = 1$, \dots, $n_k$ and $j=1$, \dots, $n_0$. 
\end{proposition}
Essentially, it is  diagonal (Hurrah!!) due to our judicious 
choice of polynomials and furthermore the identity matrix due to our choice 
of factors in Lemma~\ref{thm:psi}.

\subsection{Putting it all together}
\label{sec:finale}
Recall the expression for the kernel in~\eqref{eqn:psi-phi-sum}.
For the case $k \geq k' $,
inserting the results of Lemmas~\ref{thm:bd-inv} and~\ref{thm:psi},
\eqref{eqn:q-composition}--\eqref{eqn:qr-composition},
 taking account  of the cancellations  and letting $u\rightarrow-\infty$
gives
\begin{multline}
\label{eqn:kernel-one}
K((n_k, x, t_k), (n_{k'}, x', t_{k'})) = 
 r^{n_{k'}-n_{k}} _{t_{k'}-t_0} 
\frac{S(k)}{S(k')}
\frac{(\sigma(t_{k'}) )^{n_{k} - n_{k'} }}{\sqrt2 \sigma(t_k)}
\times \\
 \sum_{l=1}^{n_{k}}
\sqrt{
 \frac{(n_{k'} -l )!}{(n_{k} - l)!}
}
q^{(t_{k}), n_k-l}_{t_{k'}-t_{k}}
h_{n_{k}-l}^{(t_{k})}(x)
h_{n_{k'}-l}^{(t_{k'})}(x')
w^{(t_{k'})}(x').
\end{multline}
In this case you don't need to use~\eqref{eqn:phi-sum-two}. 

In the case $k < k'$ we also need 
Lemma~\ref{thm:forward-transition-expanded}.
The first $n_k$ terms in the summand in~\eqref{eqn:psi-phi-sum}
cancel the corresponding terms 
in~\eqref{eqn:forward-transition-expanded}
and we are left with the following
infinite sum.
\begin{multline}
\label{eqn:kernel-two}
K((n_k, x, t_k), (n_{k'}, x', t_{k'})) = 
-r^{n_{k'}-n_{k}} _{t_{k'}-t_0} 
\frac{S(k)}{S(k')}
\frac{(\sigma(t_{k'}) )^{n_{k} - n_{k'} }}{\sqrt2 \sigma(t_k)}
\times \\
 \sum_{l=-\infty}^{0}
\sqrt{
 \frac{(n_{k'} -l )!}{(n_{k} - l)!}
}
q^{(t_{k}), n_k-l}_{t_{k'}-t_{k}}
h_{n_{k}-l}^{(t_{k})}(x)
h_{n_{k'}-l}^{(t_{k'})}(x')
w^{(t_{k'})}(x').
\end{multline}
We shall now specialise to the 
case of Dyson BM and Warren respectively. 
\begin{proof}[Proof of Theorem \ref{thm:main-theorem-1}]

Insert the various expressions in section~\ref{sec:dyson}
into~\eqref{eqn:kernel-one} and~\eqref{eqn:kernel-two}
and multiply with the conjugating factors
\begin{equation}
\frac{S(k') e^{n'(t'-t_0) }}{S(k) e^{n(t-t_0) }}
\end{equation}
which always cancel out when you take  determinants. 

To show that the integral representation 
equals this shifted sum of Hermite polynomials, 
observe that
$$
\frac{1}{e^{-(t'-t)} v - u} \frac{v^{n'}}{u^n}   = 
\frac{v^{n'} }{u^{n+1}} \sum_{l=0}^{\infty}\left( \frac{u}{e^{-(t'-t)} v} \right)^{l} 
$$
and plug in the integral representations 
in Section~\ref{sec:hermite-representations}. 
The function $\phi$ should be expanded according 
to Lemma~\ref{thm:step-expansion}.
\end{proof}

\begin{proof}[Proof of Theorem~\ref{thm:main-theorem-2}]
Insert the various expressions in section~\ref{sec:warren}
into~\eqref{eqn:kernel-one} and~\eqref{eqn:kernel-two}
and multiply with 
\begin{equation}
\frac{S(k')}{S(k)}2^{(n'-n)/2}
\end{equation}
which always cancels out when you take  determinants. 
The integral representation is done exactly as in the last proof.
\end{proof}


 
\section{Asymptotics}
\label{sec:bead}
We shall now study the scaling limit of the Dyson Brownian Minor 
kernel from~\eqref{eqn:DB-kernel} in the bulk.

\begin{proof}[Proof of Theorem~\ref{thm:main-theorem3}]
This is done by saddle point analysis.
Let $\tau_N = e^{-(t'-t)/2N }$. 
We start out with the case $(n,t) \geq (n',t') $. 
Use the representation~\eqref{eqn:DB-kernel}
and substitute $u \mapsto u \sqrt{N/2}$ and $v \mapsto \tau_N\inv v\sqrt{N/2}$
in the integral. We must now compute
\begin{multline}
\lim_{N\rightarrow\infty}  \frac{1}{2\pi i} 
\int_\gamma du \int_\Gamma dv  \frac{\tau_N^{-N-n'-1}}{ v - u}
\frac{v^{n'}}{u^n}  \times \\
\times \exp\left[{(N/2)( \tau^{-2}_Nv^2 -u^2 +4a(u-\tau\inv_Nv)
 +2\ln v - 2\ln u) + ux  - \tau\inv_N vy} \right].
\end{multline}
The part of the exponent that 
is multiplied by $(N/2)$ is $f(u)-f(v)$
for 
\begin{align}
f(u) & = -u^2 + 4au - 2\ln u.\\
\intertext{The saddle points, satisfying the equation $f'(u)=0$  are
given by~\eqref{eqn:saddle-points} and  are both on the unit circle. 
We let $\theta $ be the angle specifying those points, 
so that $u_\pm = e^{\pm i\theta}$.
The Taylor expansions around these
critical points are }
f(u_\pm + h ) & = f(u_\pm) +( \mp 4a\sqrt{1-a^2} + 4(1-a^2)i ) h^2 + o(h^2).
\end{align}
The only thing we really need is that the coefficient of $h^2$
is non-zero. 

Now deform the contour $\Gamma$ to $\Gamma_a$ which 
is a straight line from $ a-i\infty $ to $ a+i\infty $.
See Figure~\ref{fig:contours}. 
When deforming the $v$-contour through the $u$-contour 
out pops the residue at $u= v$ which is exactly
\begin{equation}
\frac{1 }{2\pi i} \int_{u_-}^{u_+} 
u^{n'-n} e^{\frac12 (t'-t)(u^2 -2a u )  + u(x-y) } \, du.
\end{equation}
It remains to show that all the other parts of the
contours evaluate to zero. 

To see what happens around the saddle point $u_+$ 
perform the change of variables $u \mapsto u_+ + u /\sqrt{N}$
and $v \mapsto u_+ + v /\sqrt{N}$. 
\begin{multline*}
\frac{1}{2\pi i} \int du \int dv 
\frac{\tau_N^{-N-n'-1}}{v-u} \frac{(u_++v/\sqrt{N})^{n'} }{ (u_++u/\sqrt{N})^{n} } \times \\
\times
\exp\left[ \frac{1}{4} f''(u_+) (u^2-v^2)  +u_+( x -\tau^{-1} y)+
O( 1-\tau_N) + O(N^{-\frac12} ) 
\right]
\end{multline*}
which, as $N\rightarrow \infty$, tends to
\begin{multline*}
\frac{1}{2\pi i} \int du \int dv 
\frac{e^{\frac12 (t'-t) } }{v-u} (u_+)^{n'-n} 
\exp\left[ \frac{1}{4} f''(u_+) (u^2-v^2)  +u_+(x -y)
\right].
\end{multline*}
Switching $u\mapsto -u$ and $v \mapsto -v$ makes
the integrand change sign. 
Thus the integral must be zero. 
The saddle point at $u_-$ contributes zero by the same argument. 

For the  remaining contours we need to show that 
$g(u,v) := \Re f(u)-f(v) <1$. 
Well, if $g(u,v)>1$ somewhere then the integral would not 
be convergent. 
But the whole thing is a probability density and  must therefore be finite,
so this is a contradiction. 
If at some point $(u,v)$ it happens that
 $g(u,v)=1$ then $f'(u)\neq 0$ for we
have already accounted for all saddle points. 
Thus by the definition of derivative there must be 
a point nearby where $g(u,v)>1$ which is a contradiction. 

Next we need to compute the same scaling limit of the  function
$\phi$ from~\eqref{eqn:DB-kernel2} which is a very straight
forward computation, yielding~\eqref{eqn:bead2}.
\end{proof} \bibliography{article}{}
\bibliographystyle{alpha}

\end{document}